\documentclass[reqno,12pt]{article}

\usepackage{a4wide}
\usepackage{amsmath,stmaryrd} 
\usepackage{amssymb}
\usepackage{amsthm}
\usepackage[utf8]{inputenc} 
\usepackage{graphicx} 
\usepackage[english]{babel}
\usepackage{xcolor}
\usepackage[normalem]{ulem}

\numberwithin{equation}{section}

\newtheorem{theo}{Theorem}

\newtheorem{prop}{Proposition}

\newtheorem{lem}{Lemma}
\newtheorem{defi}{Definition}

\newtheorem{thmx}{Theorem}

\theoremstyle{remark}
\newtheorem{Remark}{Remark}

\newtheorem*{Remark*}{Remark}
\newtheorem*{Remarks*}{Remarks}

\makeatletter
\newcommand*{\house}[1]{%
 \mathord{%
 \mathpalette\@house{#1}%
 }%
}
\newcommand*{\@house}[2]{%
 \dimen@=\fontdimen8 %
 \ifx#1\scriptscriptstyle\scriptscriptfont
 \else\ifx#1\scriptstyle\scriptfont
 \else\textfont\fi\fi
 3 %
 \sbox0{%
 $#1%
 \vrule width\dimen@\relax
 \overline{%
 \kern2\dimen@
 \begingroup 
 #2%
 \endgroup
 \kern2\dimen@
 }%
 \vrule width\dimen@\relax
 \mathsurround=1.5\dimen@ 
 $%
 }%
 \ht0=\dimexpr\ht0-\dimen@\relax
 \dp0=\dimexpr\dp0+2\dimen@\relax
 \vbox{%
 \kern\dimen@ 
 \copy0 %
 }%
}

\newcommand{\dd}{{\rm d}}
\newcommand{\tra}{{}^t}
\newcommand{\N}{\mathbb{N}}
\newcommand{\Z}{\mathbb{Z}}
\newcommand{\Q}{\mathbb{Q}}
\newcommand{\R}{\mathbb{R}}
\newcommand{\C}{\mathbb{C}}
\newcommand{\K}{\mathbb{K}}

\newcommand{\Qbar}{\overline{\mathbb Q}}

\newcommand{\etoile}{^*}

\newcommand{\eps}{\varepsilon}
\newcommand{\ewa}{\eta}
\newcommand{\ord}{{\rm ord}}

\newcommand{\calR}{\mathcal R}

\newcommand{\calF}{\mathcal F}
\newcommand{\calS}{\mathcal S}

\newcommand{\calC}{\mathcal C}
\newcommand{\calV}{\mathcal V}
\newcommand{\Card}{{\rm Card}}

\newcommand{\rk}{{\rm rk}}

\newcommand{\ineg}{N-1\leq | \capa | \leq N}
\newcommand{\inegm}{N-2\leq | \capa | \leq N-1}
\newcommand{\capa}{\kappa}
\newcommand{\Span}{{\rm Span}}
\newcommand{\GL}{{\rm GL}}
\newcommand{\Sm}{\mathfrak S_m}
\newcommand{\Ombar}{{\overline\Omega}}

\newcommand{\unomega}{\llbracket 1, \omega \rrbracket}

\newcommand{\zerobiz}{\llbracket 0,M+t(k-1)-1 \rrbracket}
\newcommand{\zeroMmu}{\llbracket 0,M-1 \rrbracket}
\newcommand{\zeroN}{\llbracket 0,N \rrbracket}

\newcommand{\unm}{\llbracket 1,m \rrbracket}
\newcommand{\undeux}{\{1,2 \}}
\newcommand{\unr}{\llbracket 1,r \rrbracket}
\newcommand{\untheta}{\llbracket 1,\theta \rrbracket}
\newcommand{\untrois}{\llbracket 1,3 \rrbracket}
\newcommand{\quatrev}{\llbracket 4,v \rrbracket}
\newcommand{\Vz}{\mathcal V_0}
\newcommand{\vr}{\varrho}
\newcommand{\cstun}{c_1}
\newcommand{\evun}{{\rm ev}_1}

\begin{document}
\selectlanguage{english}

\title{Rational approximations to values of $E$-functions}
\date\today
\author{S. Fischler and T. Rivoal}
\maketitle

\begin{abstract} We solve a long standing problem in the theory of Siegel's $E$-functions, initiated by Lang for Bessel's function $J_0$ in the 60's and considered in full generality by G.~Chudnovsky in the 80's: we prove that irrational values taken at rational points by $E$-functions with rational Taylor coefficients have irrationality exponent equal to~2. This result had been obtained before by Zudilin under strong assumptions on algebraic independence of $E$-functions, satisfied by $J_0$ but not by all hypergeometric $E$-functions for instance. We remove them using a new generalization of Shidlovskii's lemma, analogous to zero estimates on commutative algebraic groups in which obstructions come from algebraic subgroups. 
\end{abstract}

\section{Introduction}

In rational approximation, a landmark result is Roth's theorem~\cite{roth}, proved in 1955: if $\xi\in\R\setminus\Q$ is an algebraic number, then 
\begin{equation}\label{eqmesure}
\forall \eps>0 \quad \exists c>0 \quad\forall (p,q)\in \mathbb Z\times \mathbb N^*\quad\quad
\left\vert \xi -\frac{p}{q} \right\vert \ge \frac{c}{q^{2+\eps}}.
\end{equation}
This measure of irrationality is optimal in the following sense: it would be false with $\eps=0$ and $c=1$, as can be proved using continued fractions or Dirichlet's pigeonhole principle. With respect to Lebesgue measure, almost all $\xi\in\R\setminus\Q$ satisfy~\eqref{eqmesure}. A folklore~belief~is that classical transcendental constants coming from analysis should satisfy~\eqref{eqmesure}; but this is known for very few of them, even amongst the most important ones. For instance,~\eqref{eqmesure}~is known for $\xi=\pi$ only in a weaker form, with 2 in the exponent of $q$ replaced with $7.1033$ (Zeilberger-Zudilin~\cite{zz}); in other words, the irrationality exponent of $\pi$ is at most $7.1033$. For $ \log(2)$ the situation is analogous: its irrationality exponent is at most $3.5746$ (Marcovecchio~\cite{marco}). In both cases, this exponent is currently best known and it comes after many successive improvements. 

The situation is different for the values of the exponential function: 
it is well known that~\eqref{eqmesure} holds with $\xi=e^r$ for any $r\in\Q\etoile$ (see the introduction of \cite{bakerexp} or \cite{Kappe}). This result opens the way to a possible generalization to $E$-functions, a class of functions defined by Siegel \cite{siegel} in 1929 to extend Diophantine properties of the exponential function. In this paper $\Qbar$ is seen as a subset of~$\C$.
\begin{defi}\label{defi1}
A power series $F(z)=\sum_{n=0}^{\infty} a_n z^n/n! \in \Qbar[[z]]$ is an $E$-function if 
\smallskip

\noindent $(i)$ $F$ is solution of a non-zero linear differential equation with coefficients in 
$\Qbar(z)$,

\smallskip

\noindent and there exists $C>0$ such that: 

 \smallskip

\noindent $(ii)$ for any $\sigma\in \textup{Gal}(\Qbar/\mathbb Q)$ and any $n\ge 0$, $\vert \sigma(a_n)\vert \leq C^{n+1}$.

\smallskip

\noindent $(iii)$ there exists a sequence of integers $(d_n)_{n\ge 0}$, with $1\le d_n \leq C^{n+1}$ for any $n\ge 0$, such that
$d_na_m$ are algebraic integers for all~$0\le m\le n$.

\end{defi}
Siegel's original definition is more general: in $(ii)$ and $(iii)$, he required bounds of the form ``for all $\eps>0$, $ \ldots \le n!^{\eps}$ for any $n\ge N(\eps)$'' instead of ``$ \ldots \le C^{n+1}$ for any $n\ge 0$''. 
$E$-functions in the sense of Definition~\ref{defi1} shall be called $E$-functions in the strict sense. Note that $(i)$ implies that the $a_n$'s all lie in a certain number field $\K$. An $E$-function is an entire function; it is transcendental unless it is a polynomial.

Amongst the simplest examples of $E$-functions, we mention $\exp(z)=\sum_{n=0}^\infty z^n/n!$ and Bessel's function $J_0(z):=\sum_{n=0}^\infty (iz/2)^{2n}/n!^2$. Both are 
specializations of the generalized hypergeometric function 
\begin{equation*} 
{}_pF_{q}
\left[
\begin{matrix}
a_1, \ldots, a_p
\\
b_1, \ldots, b_q
\end{matrix}
\,; z\right] := \sum_{n=0}^\infty \frac{(a_1)_n\cdots (a_p)_n}{(1)_n(b_1)_n\cdots (b_q)_n} z^n
\end{equation*}
where $q\ge p\ge 0$ and we define the Pochhammer symbol $(a)_n:=a(a+1)\cdots (a+n-1)$ for $n\ge 1$, $(a)_0:=1$. Siegel has proved that when $q\ge p\ge 0$, the parameters $a_j$ and $b_j$ are in $\mathbb Q$ (with the restriction that $b_j\notin \mathbb Z_{\le 0}$ so that $(b_j)_n\neq 0$ for all $n\ge 0$) and $c\in \Qbar$, then ${}_pF_{q}[a_1, \ldots, a_p;b_1,\ldots b_q; cz^{q-p+1}]$ is an $E$-function. However, the $\Qbar$-algebra generated by such hypergeometric $E$-functions is not large enough to contain all $E$-functions, as recently shown by Fres\'an-Jossen \cite{frejos}. We also point out that values of $E$-functions at algebraic points are closely related to exponential periods, see \cite{fresanlivre}. 

The Diophantine theory of the values taken at algebraic points by $E$-functions has a long history with classical results due, in chronological order, to Siegel, Shidlovskii, Nesterenko, Andr\'e and Beukers in particular. We refer to their original works \cite{siegel, Shidlovskiilivre, ns, andre, beukers} as well as to \cite{brs, firiliouville, Zudilin} for precise statements of these results and others. 

\medskip

Our purpose is to prove the following result, namely:~\eqref{eqmesure} holds for all irrational values of $E$-functions with rational Taylor coefficients at rational points (see \cite[Theorem 1]{Borisdec} for a consequence on the complexity of basis expansions).

\begin{theo}\label{theo:main} Let $f \in \mathbb Q[[z]]$ be an $E$-function and let $r\in \mathbb Q$.
Then either $f(r)\in\mathbb Q$, or 
for any $\eps>0$, there exists a constant $c>0$ depending on $f,r,\eps$ such that for any $p\in\Z$ and any $q\in \N\etoile$, we have 
\begin{equation}\label{eq:mesureirrat1}
\left\vert f(r)-\frac{p}{q} \right\vert \ge \frac{c}{q^{2+\eps}}.
\end{equation}
\end{theo}

Theorem~\ref{theo:main} was announced in 1984 by Chudnovsky~\cite[p.~1926,~Theorem~1 and Corollary]{chud} for $E$-functions in $\mathbb Q[[z]]$ in Siegel's original sense and in a stronger form, namely a measure of linear independence. However, as explained in \cite[p.~245]{feldnest}, the proof contains a gap in the zero estimate. Zudilin \cite{Zudilin} succeeded in filling this gap, thereby proving Theorem~\ref{theo:main} (in an even more precise form, namely with a decreasing function of $q$ instead of $\eps$), but under additional assumptions on $f$ and $r$ and for $E$-functions in the strict sense. Precisely, let $m$ be the minimal order of a non-trivial inhomogeneous differential equation with coefficients in $\mathbb Q(z)$ satisfied by $f$. Then Zudilin assumes that either $m\le 2$ (Corollary~1 on page 583) or that $f, f', \ldots, f^{(m-1)}$ are algebraically independent over $\mathbb Q(z)$ (Corollary~1 on page 557). In both cases, he also assumes that $r$ is not a singularity of the equation. 
In the present paper, we use a different approach to zero estimates (see below), which enables us to prove Theorem~\ref{theo:main} without any additional assumption on $f$. We shall first give the proof of Theorem~\ref{theo:main} for $E$-functions in the strict sense, using a new zero estimate, various lemmas in Zudilin's paper \cite{Zudilin} and results by Andr\'e and Beukers in the theory of $E$-operators for $E$-functions in the strict sense. Then we shall explain in \S\ref{genethm1} the changes that must be made to the proof to obtain Theorem~\ref{theo:main} for $E$-functions in Siegel's sense.

\medskip

As just explained, the interest of Theorem~\ref{theo:main} is that it applies to all $E$-functions with rational coefficients, for instance to all generalized hypergeometric $E$-functions of the form 
${}_pF_{q}[a_1, \ldots, a_p;b_1,\ldots b_q; c z^{q-p+1}]$
with $q\ge p\ge 0$, $c\in \mathbb Q$ and rational parameters $a_j$ and $b_j$, whereas Zudilin applies his results to them only under some assumptions on the parameters, in particular to $J_0(z)={}_1F_2[1;1,1; (iz/2)^2]$. This enabled Zudilin to answer positively a question formulated by Lang \cite{langbsmf}, namely if Theorem~\ref{theo:main} holds for $J_0$ and all $r\in \mathbb Q^*$. The $E$-function $g(z):={}_1F_2[1/2;1/3,2/3;z^2]$ does not fall under the scope of Zudilin's results because a minimal inhomogeneous differential equation satisfied by $g$ is $9 z^2g'''(z) + 9 zg''(z) - (36 z^2 + 1)g'(z) - 36 zg(z) = 0$ of order $m=3$ while $g(z)^2-g'(z)^2/4+9z^2(4g(z)-g''(z))^2/4=1$; on the other hand, $g(r)\notin \mathbb Q$ for all $r\in \mathbb Q^*$ (see \cite[\S 7]{firimathz} for details) so that~\eqref{eq:mesureirrat1} holds for all these values. See below for non-trivial examples of transcendental $_1F_1$ series with rational parameters taking a rational value at a non-zero rational point, showing that it is not possible to exclude {\em a priori} the ``$f(r)\in \mathbb Q$'' possibility even in the hypergeometric case.

Another interesting example is 
 the generating $E$-function of Ap\'ery numbers $\mathcal{A}_{2,2}(z):=\sum_{n=0}^\infty (\sum_{k=0}^n \binom{n}{k}^2\binom{n+k}{n}^2)z^n/n!$. We proved in \cite{firiliouville} that for any $r\in \mathbb Q^*$, $ \mathcal{A}_{2,2}(r)$ has irrationality exponent at most 5 (because we are not able to check that the additional conditions of Zudilin's result are met). Using Theorem~\ref{theo:main} we deduce the optimal measure~\eqref{eq:mesureirrat1} immediately for $\mathcal{A}_{2,2}(r)$.

\medskip

 A different problem, very interesting but not studied in this paper, is to know in the setting of Theorem~\ref{theo:main} whether $f(r)$ is rational or not when $f\in \mathbb Q[[z]]$ is transcendental (if~$f$ is a polynomial, obviously $f(r)\in \mathbb Q$ for all $r\in \mathbb Q$). Beukers' linear independence theorem \cite[Corollary 1.4]{beukers} provides a sufficient condition: given $\sum_{j=0}^m p_j(z)f^{(j)}(z)=q(z)$ a non-trivial inhomogeneous equation satisfied by $f$ transcendental of minimal order $m\ge 1$ with polynomial coefficients, if $r\in \mathbb Q^*$ is such that $p_m(r)\neq 0$ then $f(r)\notin \mathbb Q$. 
 If $p_m(r)=0$, it can be decided algorithmically whether $f(r)\in \mathbb Q$ or not, see \cite{bori, brs}. Of course there are trivial examples of transcendental $E$-functions such that $f(r)\in \mathbb Q$ for some $r\in \mathbb Q^*$, for instance $(z-1)e^z$ at $z=1$. But there are also non-trivial examples such as ${}_1F_1[5;7/3; -2/3]=5/27$ and ${}_1F_1[6;-2/5; -12/5]=1309/625$; see \cite[\S 4.2]{brs} for more ``exotic'' evaluations of hypergeometric $E$-functions at rational points.

 As a consequence of \cite[Theorem 4]{ateo}, if $f$ is an $E$-function with rational Taylor coefficients such that $f(1)=e^{\alpha}$ for some $\alpha \in \Qbar$, then $\alpha \in \mathbb Q$. (The proof of this result is an immediate adaptation of that of \cite[Proposition 2]{gvalues}, which is due to the referee of that paper.) In particular, Theorem~\ref{theo:main} does not apply directly to numbers such as $e^{\sqrt{2}}$, for which~\eqref{eq:mesureirrat1} is conjectural.~(\footnote{Note however that Zudilin's theorem can be applied to the $E$-function $f(x):=e^{\sqrt{2}x}+e^{-\sqrt{2}x}\in \mathbb Q[[x]]$: the number $f(1)$ has irrationality exponent 2, so that $e^{\sqrt{2}}$ has irrationality exponent at most~4, which seems to be the best known upper bound (see \cite{Kappe} and \cite{firiliouville} for proofs of the upper bounds 12 and 8 respectively).})

It would be very interesting to generalize Theorem~\ref{theo:main} to any $E$-function $f$ in $\mathbb K[[z]]$ eva\-lua\-ted at any $\alpha\in \mathbb K^*$, where $\mathbb K$ is a given fixed number field of degree $d$ over $\mathbb Q$. To our knowledge, as a consequence of \cite[Theorem 1]{firiliouville}, the current best upper bound in this generality for the irrationality exponent of $f(\alpha)$ (when it is irrational) is $d(m+1)^d$ where $m\ge 1$ is the minimal order of a non-trivial inhomogeneous differential equation with coefficients in $\Qbar(z)$ satisfied by $f$. 
Moreover, under further assumptions on $f$ and $\alpha$ (similar to Zudilin's), this bound can be largely improved using the Lang-Galochkin transcendence measure; see \cite[Theorem 5.29 and remarks]{feldnest}. Even more specifically, Kappe \cite{Kappe} has obtained the upper bound $4d^2-2d$ for the irrationality exponent of $e^\alpha$ for all $\alpha\in \Qbar^*$ of degree $d$.

\bigskip

The proof of Theorem~\ref{theo:main} is based on Chudnovsky's construction, i.e. on graded Padé approximation. Actually Zudilin used the same construction as Chudnovsky (and he provided all necessary details); so do we. The first new feature of Theorem~\ref{theo:main} is that we do not have to exclude the case where $r$ is a singularity of a differential system, because we use (variants of) results of André and Beukers on $E$-operators and $E$-functions. The second, and main, new feature is that no assumption is needed on $f$ in Theorem~\ref{theo:main}. Actually Zudilin had to make very strong assumptions in order to prove the zero estimate (namely, that a matrix consisting of values of polynomials has maximal rank -- see Proposition~\ref{propmatriceinversible} in \S \ref{construction} below). In the present paper, we prove this result using a new generalization of Shidlovskii's lemma, based on the multiplicity estimate of \cite{gfndio2} which relies on the approach of Bertrand-Beukers \cite{BB} and Bertrand \cite{DBShid}, generalized in \cite{SFcaract}. The important feature of this multiplicity estimate is that it takes into account the possibility that some exceptional solutions of the underlying differential system have identically zero remainders. Since we make no assumption on $f$ in Theorem~\ref{theo:main}, this could {\em a priori} happen and it would make it impossible to prove that the matrix we are interested in has maximal rank. 

We carry out in \S \ref{seclemzeros} a general approach to tackle this problem, in a way reminiscent of obstruction subgroups to zero estimates in commutative algebraic groups. We call {\em non-degenerate} a setting in which there is no obstruction, and prove a generalization of Shidlovskii's lemma under this assumption. We prove in \S \ref{secpreuvematriceinversible} that this property holds in the setting of the proof of Theorem~\ref{theo:main}, using an auxiliary result proved in \S \ref{sectech}. 

\bigskip

Let us also mention that all constants in the proofs of intermediate results are effective and could in principle be computed, and consequently the same can be said of the constant $c$ in Theorem~\ref{theo:main}. 

\bigskip

The structure of this paper is as follows. We gather in \S \ref{prereq} the auxiliary results we shall need: a property of the minimal inhomogeneous differential equation of an $E$-function and a desingularization lemma, adapted from results of André and Beukers respectively. We state in \S \ref{seclemzeros} the multiplicity estimate proved in \cite{gfndio2}, and derive from it new generalizations of Shidlovskii's lemma. Then \S \ref{secpreuve} is devoted to the proof of Theorem~\ref{theo:main}, using an auxiliary result stated and proved in \S \ref{sectech}. The case of $E$-functions in Siegel's sense is dealt with in \S \ref{genethm1}.

\section{Prerequisites}\label{prereq}

To begin with, we recall that the Nilsson class at $0$ is the set of functions
$$
f(z) = \sum_{e\in E} \sum_{i\in I}  
h_{e,i}(z) z^e (\log z )^i
$$
where $E \subset \C$ and $I\subset \N$ are finite subsets, and 
the functions $h_{e,i}$ are holomorphic at $0$. If such a function $f(z)$ is not identically zero, we have $E \neq \emptyset$ and we may assume that for any $e\in E$ there exists $i\in I$ such that $h_{e,i}(0)\neq 0$; then the generalized order of $f$ at $0$, denoted by $\ord_0 f$, is the minimal real part of an element $e\in E$. When $Y:\C\to\C^q$ is a vector-valued fonction, it is said to be Nilsson at 0 if its $q$ coordinates are. 

 Thoughout the paper, we shall denote by $M_n(F)$ and $M_{n,m}(F)$ the sets of $n\times n$ matrices and of $n\times m$ matrices over a given field $F$, typically $\mathbb C$ or $\mathbb C(z)$.

\subsection{Inhomogeneous differential equation of minimal order satisfied by an $E$-function}\label{subsecinhom}

Let $f$ be a transcendental $E$-function with coefficients in a number field $\K$. Consider $f_j(z)=f^{(j-1)}(z)$ for any $j\geq 1$, and denote by $m\geq 1$ the largest integer such that $1$, $f_1(z)$, \ldots, $f_m(z)$ are linearly independent over $\K(z)$. Then $f_{m+1}(z)=f^{(m)}(z)$ is a $\K(z)$-linear combination of these functions: $f$ is solution of a inhomogeneous linear differential equation of order $m$. This equation has minimal order amongst all inhomogeneous linear differential equations satisfied by $f$. Note that we consider homogeneous equations to be special cases of inhomogeneous equations for which the constant term is equal to 0, and it may happen that an inhomogeneous equation of minimal order for one of its solutions is in fact homogeneous.

\begin{prop} \label{propinhom}
The point 0 is either a regular point, or a regular singularity, of any inhomogeneous linear differential equation of minimal order satisfied by a transcendental $E$-function.
\end{prop}
André has proved this result for the homogeneous linear differential equation of minimal order (and even a more precise one, see \cite[Théorème de pureté, p. 706]{andre}). In this proposition, it is understood that an inhomogeneous equation is regular or regular singular at $0$ if the companion differential system $Y'=AY$ with solution ${}^t(1,f,f',\ldots, f^{(m-1)})$ is.

\begin{proof} Let $f$ be a transcendental $E$-function solution of a non-trivial homogeneous differential equation $Ly=0$ of minimal order $r\ge 1$, where $L\in \mathbb C(z)[d/dz]$. A minimality argument implies that a non-trivial minimal inhomogeneous equation satisfied by $f$ has order $r$ or $r-1$ (see \cite[\S 4]{bori}). 

In the former case, this inhomogeneous equation is then simply equal to $L$ up a non-zero rational function factor, and the claim follows because, by Andr\'e's above-mentioned theorem, $0$ is a regular point or a regular singularity of $L$.

We now deal with the case of order $r-1$ which is more complicated. 
In this situation, there exists $R\in \mathbb C(z)\setminus \{0\}$ such that $L^*R=0$, where $L^*$ is the adjoint of $L$~(\footnote{Given a differential operator $L= \sum_{j=0}^r q_j(d/dz)^j\in \mathbb C(z)[d/dz]$, its adjoint $L^*\in \mathbb C(z)[d/dz]$ is defined by $L^*y=\sum_{j=0}^r (-1)^{j}(q_jy)^{(j)}$; see \cite[p.~38]{poole}.}) and a minimal inhomogeneous equation satisfied by $f$ is of the form
\begin{equation}\label{eq:inhomeqfc}
\sum_{j=0}^{r-1} p_j(z)f^{(j)}(z)=c
\end{equation}
where 
\begin{equation}\label{eq:inhomeqfc2}
\frac{d}{dz}\left(\sum_{j=0}^{r-1} p_j(z) \frac{d^j}{dz^j}\right) = RL.
\end{equation}
The rational function $R$ can be explicitly determined from $L$, and then Eq.~\eqref{eq:inhomeqfc2} enables to determine suitable $p_j\in \mathbb C(z)$ from $R$ and $L$ (first $p_{r-1}$, then $p_{r-2}$, etc). Finally, the constant $c$ is computed by determining the constant term in the Laurent series expansion of the left-hand side of~\eqref{eq:inhomeqfc}. See \cite[\S 2.4]{brs} for more details. By minimality of $L$, $c\neq 0$ and without loss of generality we can assume that $c=1$. 
Since $f$ is transcendental, we have $r\ge 2$.
Let $f_2, \ldots, f_{r}$ be other local solutions of $Ly=0$ at $z=0$ such that $f_1:=f, f_2, \ldots,f_{r}$ make up a $\mathbb C$-basis of the vector space of solutions of $Ly=0$: by Andr\'e's theorem, each $f_k$ is in the Nilsson class at $z=0$ 
because $0$ is at worst a regular singularity of $L$. From Eq.~\eqref{eq:inhomeqfc2}, we observe that 
$$
\frac{d}{dz}\left(\sum_{j=0}^{r-1} p_j(z)f_k^{(j)}(z)\right) = R(z)Lf_k(z)=0 \quad \mbox{for any } k \in\unr
$$
so that 
$$
\sum_{j=0}^{r-1} p_j(z)f_k^{(j)}(z) = c_k
$$
for some $c_k\in \mathbb C$ (and $c_1=1$). Up to reordering the basis $f_1, f_2, \ldots,f_{r}$ and multiplying each $f_k$ by a non-zero constant, we can and shall assume without loss of generality that $f_1, \ldots, f_s$ are such that $c_k=1$ for $k=1, \ldots, s$ and $f_{s+1}, \ldots, f_r$ are such that $c_k=0$ for $k=s+1, \ldots, r$, for some $s\in \unr$. We now write the inhomogeneous equation $\sum_{j=0}^{r-1} p_j(z)y^{(j)}(z) = 1$ satisfied by $f$ as a companion differential system $Y'=AY$ where $A\in M_r(\mathbb C(z))$, with the vector solution 
${}^t(1, f, \ldots, f^{(r-2)})$. From what precedes, it turns out that in fact 
$$
U:=\left(\begin{matrix}
1 &\cdots &1&0&\cdots &0
\\
f_1&\cdots &f_s&f_{s+1}&\cdots &f_r
\\
\vdots& \cdots & \vdots&\cdots &\vdots&\cdots
\\
f_1^{(r-2)}&\cdots&f_s^{(r-2)}&f_{s+1}^{(r-2)}&\cdots &f_r^{(r-2)}
\end{matrix}\right)
$$
is a fundamental matrix solution of $Y'=AY$. (By definition, if $s=r$, there are only 1's on the first line of $U$). Indeed, the columns of $U$ are solutions of $Y'=AY$ and they are $\mathbb C$-linearly independent because on the second line $f_1, \ldots, f_r$ are $\mathbb C$-linearly independent, so that $U$ is invertible. 
Since the entries of $U$ are in the Nilsson class at $z=0$, it follows from \cite[ 
p.~81]{putsinger} that $0$ is a regular point or a regular singular point of $Y'=AY$.
\end{proof}

\subsection{A version of Beukers' desingularization lemma}\label{subsecdesing}

Beukers' desingularization lemma \cite[Theorem 1.5]{beukers} is very useful when dealing with $E$-functions, since it enables one to get rid of all non-zero singularities of the underlying (homogeneous) differential equation. In this section we state and prove a non-homogeneous version of this result that incorporates several additional features: the coefficients lie in a fixed number field (as in \cite[Proposition 2]{firiliouville}), and two useful properties are preserved (the value at 1 of the first $E$-function, and the property that $0$ is a regular singularity). These properties will be very important in the proof of Theorem~\ref{theo:main}.

\begin{prop} \label{propbeukers}
Let $g_1$, \ldots, $g_m$ be $E$-functions with coefficients in a number field $\K$, such that $1$, $g_1$, \ldots, $g_m$ are linearly independent over $\C(z)$. Assume that the column vector $ (1, g_1, \ldots, g_m)$ is solution of a first-order differential system $Y'=SY$ with $S\in M_{m+1}(\K(z))$.

Then there exist $E$-functions $f_1$, \ldots, $f_m$ with coefficients in $\K$ such that:
\begin{itemize}
 \item The functions 1, $f_1$, \ldots, $f_m$ are linearly independent over $\C(z)$. 
 \item There exist polynomials $Q_{j,l}(z)\in \K[z]$ such that $g_j(z) = \sum_{l=0}^m Q_{j,l}(z) f_l(z)$ for any $j\in\unm$, where we let $f_0(z)=1$.
 \item The column vector $ (1, f_1, \ldots, f_m)$ is solution of a first-order differential system $Y'=\widetilde SY$ with $\widetilde S\in M_{m+1}(\K[z,1/z])$.
 \item If $g_1(1) $ is transcendental then $g_1(1)=f_1(1)$.
 \item If 0 is a regular singularity of the system $Y'=SY$, then it is also a regular singularity of the system $Y'=\widetilde SY$.
 \end{itemize}
\end{prop}

\begin{proof}
It follows closely that of Beukers, or more precisely the version over a 
number field $\K$ given in \cite[Proposition 2]{firiliouville} (see also \cite{brs}). This approach consists in finitely many steps. At each step, one obtains a $\K$-linear combination of $1$, $g_1$, \ldots, $g_m$ that vanishes at some non-zero algebraic point $\alpha$; then one replaces one of the functions by this linear combination, and divides by $z-\alpha$ if $\alpha\in\K $ (by the minimal polynomial of $\alpha$ over $\K$ in the general case). The point is that the above-mentioned linear combination is never just the function 1, since 1 does not vanish at $\alpha$. Therefore the function 1 can be preserved at each step. Moreover it is clear from the proof that 0 being a regular singularity holds throughout the procedure. This proves Proposition~\ref{propbeukers}, except for the property $g_1(1)=f_1(1)$. However, if $g_1(1) = \sum_{l=0}^m Q_{1,l}(1) f_l(1)$ is transcendental then there exists $l_0\in\unm$ such that $ Q_{1,l_0}(1) \neq 0$. Replacing $f_{l_0}(z)$ with $\sum_{l=0}^m Q_{1,l}(1) f_l(z)$ does not change the other properties of the functions $f_1$, \ldots, $f_m$, and provides $g_1(1)=f_{l_0}(1)$. Up to a permutation of $f_1$, \ldots, $f_m$ this concludes the proof of Proposition~\ref{propbeukers}.
\end{proof}

\section{A new version of of Shidlovskii's lemma} \label{seclemzeros}

\newcommand{\cdeux}{c_2}
\newcommand{\ctrois}{c_3}
\newcommand{\calRP}{\calR[P_1,\ldots,P_q]}

In this section we state and prove a new version of Shidlovskii's lemma, that we shall use in the proof of Theorem~\ref{theo:main} (see \S \ref{secpreuvematriceinversible}). The original result of Shidlovskii~\cite[Chapter~3, \S 7, Lemma 10]{Shidlovskiilivre} has been generalized, with a different proof, by Bertrand-Beukers \cite{BB} using differential Galois theory and a generalization of Fuchs' relation on exponents of a differential equation. Further generalizations allow vanishing conditions involving several solutions at several points (see \cite{DBShid, SFcaract, gfndio2}). 

In the present paper we are interested in the case where the remainders of many solutions $Y$ of a differential system $Y'=AY$ vanish at $0$ with high multiplicity. The difficult point is that some of these remainders may be identically zero. We have tackled this problem in \cite{gfndio2}, but this result (stated in \S \ref{subseclemzeroquebec} below) is not ready-to-use because it involves a condition on the polynomials $P_i$, on which no information is available in general. In this section we state and prove two versions of Shidlovskii's lemma that can be used much more easily in practice. The first one (see \S \ref{subsecobstru}) involves obstructions, and is reminiscent of zero estimates in commutative algebraic groups. The second one (see \S \ref{subsecnondeg}) deals with the case where no such obstruction exists; we call {\em non-degenerate} such a situation. We shall prove in \S \ref{secpreuvematriceinversible} that the setting of the proof of Theorem~\ref{theo:main} is non-degenerate over $\Qbar$, using an auxiliary result proved in~\ref{sectech}.

\subsection{Setting and known results} \label{subseclemzeroquebec}

Let $q\geq 1$, $A \in M_q(\C(z))$, $n\geq 0$, and $P_1,\ldots,P_q\in\C[z]$ be such that 
 $\deg P_i \leq n$ for any $i$. 
We identify tuples in $\C^q$ with column matrices in $M_{q,1}(\C)$.
 Then with any solution $Y = (y_1,\ldots,y_q)$ of the differential system $Y'=AY$ is associated a remainder $R(Y)$ defined by 
$$
R(Y)(z) = \sum_{i=1}^q P_i(z) y_i(z).$$
The derivatives of such a remainder can be written as \cite[Chapter 3, \S 4]{Shidlovskiilivre}
\begin{equation} \label{eqderiR}
R(Y)^{(k-1)}(z) = \sum_{i=1}^q P_{k,i}(z) y_i(z),
\end{equation}
where 
 the rational functions $P_{k,i}\in\C(z)$ are defined for $k \geq 1$ and $1 \leq i \leq q$ by
\begin{equation} \label{eqdefpki}
\left( \begin{array}{c} P_{k,1} \\ \vdots \\ P_{k,q} \end{array}\right) = \left(\frac{\dd}{\dd z} + \, \tra A\right)^{k-1} 
 \left( \begin{array}{c} P_{1} \\ \vdots \\ P_{q} \end{array}\right).
 \end{equation}
Obviously the poles of the $P_{k,i}$ of $M$ are amongst the singularities of the differential system $Y'=AY$.

\medskip

The main new feature of the multiplicity estimate proved in \cite{gfndio2} is that it takes into account the possibility that $R(Y)(z)$ is identically zero for some non-zero solutions $Y(z)$ of the differential system $Y'=AY$. To state this result (in the special case we are interested in), we denote by $\varrho\geq 0 $ the dimension of the $\C$-vector space of solutions $Y$ such that $R(Y)(z)$ is identically zero.

\begin{thmx}
\label{thshid}
There exists a positive constant $\cstun$, which depends only on $A$, with the following property. 
Let 
$(Y_j)_{j\in J}$ be a family of solutions of $Y'=AY$ such that the functions $R(Y_j)$, $j\in J $, are $\C$-linearly independent and 
belong to the Nilsson class at $0$.
Assume that 
\begin{equation} \label{eqhypdetnn}
 \sum_{j\in J } \ord_0(R(Y_j)) \geq (n+1) ( q - \varrho) -\tau 
\end{equation}
 for some $\tau\in\Z $.
Then:
 \begin{itemize}
 \item[$(i)$] We have $\tau \geq -\cstun$.
 \item[$(ii)$] If $0\leq \tau \leq n - \cstun$ then for any $\alpha\in\C\etoile$ which is not a singularity of the differential system $Y'=AY$, the matrix $(P_{k,i}(\alpha))_{1\leq i \leq q, 1\leq k < \tau + \cstun} \in M_{q,\tau + \cstun-1}(\C)$ has rank at least $q-\varrho$.
 \end{itemize}
 \end{thmx}

In this setting, under the assumptions of $(ii)$, the matrix $(P_{k,i}(\alpha))$ has rank equal to $q-\varrho$. Indeed, there exist $\varrho$ $\C$-linearly independent solutions $Y$ such that $R(Y)$ is identically zero. For each of them, we have $\sum_{i=1}^q P_{k,i}(\alpha)y_i(\alpha)= R(Y)^{k-1}(\alpha)=0$ for any $k\geq 1$: since $\alpha$ is not a singularity, this provides $\varrho$ linearly independent linear relations between the rows of the matrix $(P_{k,i}(\alpha))$.

We remark that Theorem~\ref{thshid} would not hold if the linear independence assumption were on the $Y_j$ rather than on the $R(Y_j)$. Indeed, for instance if $R(Y_j) $ were identically zero for some $j\in J$, then $ \ord_0(R(Y_j))$ would be infinite and Eq.~\eqref{eqhypdetnn} would hold for any $\tau\in\Z$. Moreover, solutions $Y$ such that $R(Y)$ is identically zero are taken advantage of in Theorem~\ref{thshid}: each such solution (in a linearly independent family) provides the same benefit in Eq.~\eqref{eqhypdetnn} as an additional function $Y_j$ such that $R(Y_j)$ would vanish to order $n+1$. 

\begin{Remark}\label{remeff0} Following the proof \cite{gfndio2} of Theorem~\ref{thshid} and using the results of \cite{BCY} shows that $\cstun$ can be effectively computed in terms of $A$.
\end{Remark}

\subsection{A multiplicity estimate with obstruction} \label{subsecobstru}

In this section we state and prove a new multiplicity estimate, namely Theorem~\ref{thobstru}, reminiscent of those on commutative algebraic groups (see below).

Let $q\geq 1$, $A \in M_q(\C(z))$, $n\geq 0$, and $P_1,\ldots,P_q\in\C[z]$ be such that 
 $\deg P_i \leq n$ for any $i$. 
We denote by $\calS$ the set of solutions of the differential system $Y'=AY$, and by $\calR$ the set of all $Y\in\calS$ whose remainder $R(Y)(z) = \sum_{i=1}^q P_i(z) y_i(z)$ is identically zero.

\begin{theo}
\label{thobstru}
There exists a positive constant $\cstun$, which depends only on $A$, with the following property. 
Let $\cdeux, T\geq 0$ and $\calF$ be a subspace of $\calS$ such that for any $Y\in \calF$, the associated remainder $R(Y)$ belongs 
 to the Nilsson class at $0$ and vanishes at $0$ with order at least $T$. Assume that 
 \begin{equation} \label{eqhypobstru}
(n+1)(q-\dim \calR) \leq T (\dim\calF - \dim(\calF\cap\calR))+\cdeux.
\end{equation}
Then:
 \begin{itemize}
 \item We have $\cdeux \geq -\cstun$.
 \item If $\cdeux \leq n - \cstun$ then for any $\alpha\in\C\etoile$ which is not a singularity of the differential system $Y'=AY$, the matrix $(P_{k,i}(\alpha))_{1\leq i \leq q, 1\leq k <\cstun+\cdeux} \in M_{q, \cstun+\cdeux-1}(\C)$ has rank equal to $q- \dim\calR$.
 \end{itemize}
 \end{theo}

In this result the constant $\cstun$ is the same as in Theorem~\ref{thshid}. 
We notice that Eq.~\eqref{eqhypobstru} can be written as 
$$
(n+1)\dim (\calS / \calR) \leq T \, \dim\calF/ (\calF\cap\calR)+\cdeux .
$$
It is similar to the inequalities that appear in zero estimates on commutative algebraic groups (see for instance \cite{MW, Pph, Wu}), where $\calS$ corresponds to an algebraic group $G$, $\calF$ to an analytic subgroup $W$ (along which derivatives are taken) or a discrete subgroup $\Gamma$ (i.e. points at which values are taken), and $\calR$ to an algebraic subgroup of $G$. In this setting an algebraic subgroup of $G$ constitutes an obstruction if it contains ``more elements'' of $\Gamma$ or $W$ than it should (in terms of its dimension).

\bigskip

\begin{proof}[Proof of Theorem~\ref{thobstru}] We write $\varrho=\dim\calR$, $\chi = \dim(\calF\cap\calR)$, and $\theta=\dim\calF$. There exists a basis $(Y_1,\ldots, Y_\theta)$ of $\calF$ such that $(Y_1,\ldots, Y_\chi)$ is a basis of $ \calF\cap\calR$. Then we claim that $R(Y_{\chi+1}),\ldots, R(Y_\theta)$ are linearly independent over $\C$.

Indeed, let $\lambda_{\chi+1},\ldots, \lambda_\theta\in\C$ be such that $ \lambda_{\chi+1} R(Y_{\chi+1}) + \ldots + \lambda_\theta R(Y_\theta)$ is identically zero. Then $Y = \lambda_{\chi+1} Y_{\chi+1} + \ldots + \lambda_\theta Y_\theta$ belongs to $ \calF\cap\calR$: it can be written as 
$\lambda_1 Y_1 + \ldots + \lambda_{\chi} Y_\chi$ for some $\lambda_{1},\ldots, \lambda_\chi\in\C$. This implies $\lambda_1 Y_1 + \ldots + \lambda_{\chi} Y_\chi - \lambda_{\chi+1} Y_{\chi+1} - \ldots - \lambda_\theta Y_\theta = 0$ so that all $\lambda_j $ are zero: the claim is proved.

We let 
$$
\tau = \max(0, (n+1)(q-\varrho) - T(\theta-\chi) )
$$
so that Eq.~\eqref{eqhypobstru} yields $\tau\leq \cdeux$ since $\cdeux\geq 0$. Since $\ord_0(R(Y_j)) \geq T$ for any $j$ we have 
$$\sum_{j=\chi+1}^\theta \ord_0(R(Y_j)) \geq (\theta-\chi) T \geq (n+1)(q-\varrho) -\tau,$$
and the functions $R(Y_j)$, $\chi+1\leq j \leq \theta$, are linearly independent over $\C$. Therefore Theorem~\ref{thshid} yields $\cdeux\geq \tau \geq -\cstun$ where $\cstun$ is the constant in that theorem, and if $ \cdeux\leq n - \cstun$ then for any $\alpha\in\C\etoile$ which is not a singularity of the differential system $Y'=AY$, the matrix $(P_{k,i}(\alpha))_{1\leq i \leq q, 1\leq k < \tau + \cstun} \in M_{q,\tau + \cstun-1}(\C)$ has rank at least $q-\varrho$. Since $\tau\leq \cdeux$ the same conclusion holds with $\cdeux$ instead of $\tau$. At last the rank is equal to $q-\varrho$ using the remark that follows Theorem~\ref{thshid}.
\end{proof}

\subsection{The nondegenerate setting} \label{subsecnondeg}

Let $\K $ be a subfield of $\C$, and $A \in M_q(\K(z))$ with $q\geq 1$. We denote by $\calS$ the $q$-dimensional vector space (over $\C$) of solutions of the differential system $Y'=AY$. We consider a non-zero subspace $\calF$ of $\calS$.

For $P_1,\ldots,P_q\in\C[z]$ we denote by $\calRP$ the space of solutions $Y = (y_1,\ldots,y_q)\in\calS$ such that $R(Y)$ is identically zero, where $R(Y)(z) = \sum_{i=1}^q P_i(z) y_i(z)$. This subspace of $\calS$ is said to be {\em proper} if it is distinct from $\{0\}$ and from $\calS$.

\begin{defi} The subspace $\calF$ is said to be {\em nondegenerate over $\K$} if for any $P_1,\ldots,P_q\in\K[z]$ such that $\calRP$ is proper, we have
 \begin{equation} \label{eqstrict}
\frac{\dim(\calF\cap\calRP)}{\dim\calRP} < \frac{\dim \calF }{\dim \calS}.
\end{equation}
\end{defi}

In loose terms, this means that points of $\calF$ do not accumulate in any such $\calRP$. 
We point out that the coefficients of $P_1,\ldots,P_q$ are assumed to belong to $\K$, but no such rationality assumption is made on $\calF$. 
We refer to \cite[\S 1.3]{Asterisque} for general properties of distribution of discrete subgroups with respect to algebraic subgroups, in which similar inequalities occur.

Since $\dim\calS=q$, an equivalent formulation of Eq.~\eqref{eqstrict} is 
 \begin{equation} \label{eqformulationeq}
\calF \not\subset \calRP \quad \mbox{ and } \quad  \frac{q}{\dim \calF} > \frac{q-\dim\calRP}{\dim\calF-\dim(\calF\cap\calRP)}.
\end{equation}

\bigskip

The point in asking $\calF$ to be nondegenerate over $\K$ is that we obtain the following version of Shidlovskii's lemma, in which no assumption involves $\calRP$.

\begin{theo}
\label{thnondeg}
Assume that $\calF$ is nondegenerate over $\K$. 

Let $0\leq \eps \leq \frac1{q+1}$, $w\in\R_{\geq 0}$, and $n\in\N$; assume that $n$ is sufficiently large in terms of $w$ and the differential system $Y'=AY$. Let $P_1,\ldots,P_q\in\K_n[z]$ be polynomials of degree at most $n$, not all zero, with coefficients in $\K$, such that for any $Y = (y_1,\ldots,y_q)\in\calF$ the remainder $R(Y)(z) = \sum_{i=1}^q P_i(z) y_i(z)$ belongs 
 to the Nilsson class at $0$ and vanishes at $0$ with order at least $ \frac{(q-\eps)n}{\dim\calF} - w$. 
 
 Then $\calRP=\{0\}$ and for any $\alpha\in\C\etoile$ which is not a singularity of the differential system $Y'=AY$, the matrix $(P_{k,i}(\alpha))_{1\leq i \leq q, 1\leq k < \cstun + (w+1)q + n\eps} \in M_{q, \cstun+ (w+1)q + n\eps-1}(\C)$ has rank $q $ (where $\cstun$ is a constant that depends only on $A$). 
 \end{theo}

Here the constant $\cstun$ is the same as in Theorems~\ref{thshid} and~\ref{thobstru}.

An important point here is that $\calRP=\{0\}$: it turns out that $0$ is the only solution $Y\in\calS$ such that $R(Y)$ is identically zero. This is necessary for the matrix $(P_{k,i}(\alpha))$ to have maximal rank equal to $q$, since any such solution $Y$ yields a linear relation between its rows (see the remark after Theorem~\ref{thshid}).

\bigskip

To apply Theorem~\ref{thnondeg} one needs to prove that $\calF$ is non-degenerate over $\K$. In some settings this can be done using differential Galois theory, because all subspaces $\calRP$ are stable under the differential Galois group of $Y'=AY$. In the present paper we shall proceed in a different way (see \S \ref{secpreuvematriceinversible}). Indeed we shall prove (see Theorem~\ref{thrat} in \S \ref{sectech})
that Eq.~\eqref{eqstrict} holds for any proper subspace of $ \C^q$ defined over $\Qbar$, up to identifying a solution and its value at $\alpha=1$, and apply the following lemma. 

\begin{lem}\label{lemrrationnel}
 Let $\alpha\in\K\etoile$, not a singularity of the differential system $Y'=AY$. Then for any $P_1,\ldots,P_q\in\K[z]$, the image of $\calRP$ under the isomorphism $\calS \to \C^q$, $Y\mapsto Y(\alpha)$ is a subspace of $ \C^q$ defined over $\K$.
\end{lem}

By {\em defined over $\K$}, or {\em rational over $\K$}, we mean that this subspace of $\C^q$ has a $\mathbb C$-basis consisting of vectors in $\K^q$. This is equivalent to the existence of a system of linear equations with coefficients in $\K$ that defines it.

\medskip

\begin{proof}[Proof of Lemma~\ref{lemrrationnel}] Let $Y\in \calS$. Then $Y(\alpha)$ belongs to the image of $\calRP$ if, and only if, $Y\in\calRP$, that is $\sum_{i=1}^q P_i(z) y_i(z)$ is identically zero. This is equivalent to the fact that all derivatives of this function vanish at $\alpha$, i.e. $\sum_{i=1}^q P_{k,i}(\alpha) y_i(\alpha) = 0$ for any $k\geq 1$. This is a family of linear forms in the coordinates of $Y(\alpha)$, with coefficients in $\K$ (since $\alpha\in\K$, $A \in M_q(\K(z))$ and $P_1,\ldots,P_q\in\K[z]$). This concludes the proof of Lemma~\ref{lemrrationnel}.
\end{proof}

\bigskip

\begin{proof}[Proof of Theorem~\ref{thnondeg}] We shall apply Theorem~\ref{thobstru} twice with $T= \frac{(q-\eps)n}{\dim\calF} -w$; for simplicity we write $\calR = \calRP$.

Since $P_1,\ldots,P_q$ are not all zero, there exist $i_0\in\{1,\ldots,q\}$ and $z_0\in\C$ (outside the singularities of $Y'=AY$) such that $P_{i_0}(z_0)\neq 0$. Then there exist $v_1,\ldots,v_q\in\C$ such that $\sum_{i=1}^q P_i(z_0)v_i\neq 0$, and a solution $Y$ of $Y'=AY$ such that $y_i(z_0)=v_i$ for any $i$. Then the corresponding remainder $R(Y)$ does not vanish at $z_0$, so that $Y\not\in\calR$ and $\calR\neq \calS$.

Let us assume that $\calR\neq \{0\}$, so that $\calR$ is a proper subspace of $\calS$. Then Eq.~\eqref{eqformulationeq} yields 
$$
\calF \not\subset \calR \quad \mbox{ and } \quad
\frac{q}{\dim \calF} \geq \frac{q-\dim\calR}{\dim\calF-\dim(\calF\cap\calR)}+
 \frac1{(\dim\calF)^2}
$$
since the difference between the two fractions of Eq.~\eqref{eqformulationeq} is a positive rational number with denominator at most $(\dim\calF)^2$.
This implies
$$
\frac{q-\eps}{\dim \calF} \geq\frac{q-\dim\calR}{\dim\calF-\dim(\calF\cap\calR)}+
\frac1{q(q+1) \dim\calF }
$$
since $\frac1{\dim \calF}-\eps \geq \frac1q - \frac1{q+1}=\frac1{q(q+1)}$, and therefore 
$$ \frac{q-\eps}{\dim \calF} ( \dim \calF - \dim(\calF\cap\calR)) \geq q - \dim \calR + \frac{\dim \calF - \dim(\calF\cap\calR)}{ q(q+1) \dim\calF}
 \geq q - \dim \calR + \frac1{q^2(q+1)}$$
 using the fact that $\calF \not\subset \calR$. 
Since $n$ is sufficiently large in terms of $w$ and the differential system $Y'=AY$ we obtain
$$
(n+1)(q-\dim \calR) \leq \Big( \frac{(q-\eps)n}{\dim\calF} - w \Big) (\dim\calF - \dim(\calF\cap\calR)) -\cstun - 1
$$
where $\cstun$ is the constant in Theorem~\ref{thobstru}. This is a contradiction with the first conclusion of Theorem~\ref{thobstru}, so that actually $\calR = \{0\}$.

Therefore letting $\cdeux = q +w\dim\calF + n \eps $ we have 
\begin{eqnarray*}
 (n+1)(q-\dim \calR) = (n+1) q &\leq& \Big( \frac{(q-\eps)n}{\dim\calF} - w\Big) \dim\calF+ \cdeux \\
 &=& \Big( \frac{(q-\eps)n}{\dim\calF} - w \Big) (\dim\calF - \dim(\calF\cap\calR))+\cdeux.
 \end{eqnarray*}
Since $n$ is sufficiently large in terms of $w$ and the differential system $Y'=AY$ we have $\cdeux \leq n - \cstun$: Theorem~\ref{thobstru} shows that for any $\alpha\in\C\etoile$ which is not a singularity of the differential system $Y'=AY$, the matrix $(P_{k,i}(\alpha))_{1\leq i \leq q, 1\leq k <\cstun+q +w\dim\calF + n \eps } \in M_{q, \cstun+q +w\dim\calF + n \eps -1}(\C)$ has rank $q$.  
This concludes the proof of Theorem~\ref{thnondeg}.
\end{proof}

\section{Proof of the main result} \label{secpreuve}

This section is devoted to the proof of Theorem~\ref{theo:main}, using an auxiliary result that will proved later in \S \ref{sectech}.

In \S \ref{secnotations} we introduce the notation and setting of the proof, applying the results of \S \S \ref{subsecinhom} and~\ref{subsecdesing}. Then in \S \ref{construction} we give construction and properties of graded Padé approximants, including the zero estimate (namely Proposition~\ref{propmatriceinversible}). 
Admitting this result, we prove 
Theorem~\ref{theo:main} stated in the introduction in \S \ref{proofthm1}, and its generalization to an $E$-function in Siegel's orginal sense with coefficients in $\mathbb Q$, in \S \ref{genethm1}.

The rest of the paper is devoted to the proof of Proposition~\ref{propmatriceinversible}, carried out in \S \ref{secpreuvematriceinversible} using the differential system considered in \S \ref{secsystemediff} and the auxiliary result proved in \S \ref{sectech}.

\subsection{Setting, notations and parameters}\label{secnotations}

In this section we describe the setting of the proof of Theorem~\ref{theo:main}. We use the results of \S \S \ref{subsecinhom} and~\ref{subsecdesing} to obtain a family of linearly independent $E$-functions, solution of a first order differential system with no non-zero finite singularity. Then we introduce (essentially) the same notation and parameters as in Zudilin's proof.

\medskip

To prove Theorem~\ref{theo:main} we start with an $E$-function $g(z)$ with coefficients in $\Q$, and $r\in\Q\etoile$. We assume that $g(r)$ is irrational; then $g(r)$ is transcendental by \cite[Theorem 4]{ateo}. Considering $g(rz)$ instead of $g$, we may assume that $r=1$. 

As in \S \ref{subsecinhom} we consider a (possibly) inhomogeneous linear differential equation of minimal order satisfied by the transcendental $E$-function $g$. Proposition~\ref{propinhom} asserts that 0 is (at worst) a regular singularity of this equation. Viewing this equation as a differential system of order one satisfied by the column vector $ (1, g ,g', \ldots, g^{(m-1)})$, we apply Proposition~\ref{propbeukers} and obtain $E$-functions 1, $f_1$, \ldots, $f_m$ with coefficients in $\Q$, linearly independent over $\C(z)$, such that $f_1(1)=g(1)$ is the number we are interested in to prove Theorem~\ref{theo:main}. The important point is that $ (f_1 , \ldots, f_m)$ is a solution of a first-order inhomogeneous differential system
\begin{equation} \label{eqsysdiffinitial}
 f'_l(z) = S_{l,0} (z) + \sum_{j=1}^m S_{l,j}(z) f_j(z) \mbox{ for any } l\in\unm 
\end{equation} 
with $ S_{l,j}(z) \in \Q[z,1/z]$: the only possible finite singularity of this system is zero, and it is regular (in case it is a singularity).

\medskip

We consider multi-indices $ \capa\in\N^m$ and sums over such multi-indices. In such a sum, whenever an index $\capa$ belongs to $\Z^m$ but not to $\N^m$ (i.e., has at least one negative component), the term corresponding to this index will be considered as 0. For $\capa=(\capa_1,\ldots,\capa_m)\in\N^m$, we write $|\capa| = \capa_1+\ldots+\capa_m$.

We denote by $(e_1,\ldots,e_m)$ the canonical basis of $\Z^m$, i.e. $e_i=(0,\ldots,0,1,0,\ldots,0)$ where the $i$-th coordinate is equal to 1.
We let 
$$\Omega = \{\capa\in\N^m, \, \ineg\}, 
\quad
\Theta = \{\capa\in\N^m, \, |\capa| = N\},
$$
and 
$$
\omega=\Card\, \Omega = \binom{N+m-2}{m-1} + \binom{N+m-1}{m-1} , 
\quad
\theta = \Card \, \Theta = \binom{N+m-1}{m-1}.
$$
We remark, for future reference, that 
$$\frac{\omega}{\theta} = 2 - \frac{m-1}{N+m-1} = 1+ \frac{N}{N+m-1}.$$
We also fix a bijective map $\unomega\to\Omega$ so that indices $\capa\in\Omega $ can be seen as integers between 1 and $\omega$; for instance a family $(x_\omega)_{\omega\in\Omega}\in\C^\Omega$ can also be seen as a tuple in $\C^\omega$, or as a column matrix in $M_{\omega,1}(\C)$.

\bigskip

Let us introduce now the parameters that will be used in the construction.
Let $N$ be sufficiently large with respect to $f_1$, \ldots, $f_m$, and let $\ewa>0$ be a real number such that 
\begin{equation}\label{eqcst1}
 \ewa \leq \frac1{3(N+m-1)}.
\end{equation}
This number $\ewa$ plays the role of the real denoted by $\eps$ in \cite{Zudilin}. At the end of the proof (namely in \S \ref{secpreuvematriceinversible} when Theorem~\ref{thnondeg} is applied), we shall assume also $\ewa\leq \frac1{\omega+1}$. 
Let $M$ be sufficiently large with respect to $f_1$, \ldots, $f_m$, $N$, and $\ewa$; consider 
\begin{equation}\label{eqcst2}
K = \Big\lfloor \frac{(\omega-\ewa)M}{\theta}\Big\rfloor. 
\end{equation}
 
\begin{Remark}\label{remeff1}
These parameters are the same as the ones used by Zudilin, except that he assumes that equality holds in Eq.~\eqref{eqcst1}, and gives an explicit value for $M$ in terms of $N$. The reason for this difference is that we have not computed explicitly the constant $C_4$ in Proposition~\ref{propmatriceinversible} (which depends on $N$). This is useless for our purpose but it could be done (see Remark~\ref{remeff2}); it would lead to explicit values of $\ewa$ and $M$, and then to an explicit value of 
 the constant $c$ in Theorem~\ref{theo:main}.
\end{Remark}

\subsection{Construction and properties of graded Padé approximants} \label{construction}

In this section, we state the construction and properties of graded Padé approximants, sketched by Chudnovsky \cite{chud} and proved in detail by Zudilin \cite{Zudilin}. Apart from the zero estimate (namely Proposition~\ref{propmatriceinversible} below), the results are exactly the same as in Zudilin's paper. 

To motivate this construction, let us explain it with different notations in the case $m=2$. We shall construct polynomials $A_0,\ldots,A_N,B_0,\ldots,B_{N-1}$ such that $A_i(z)+B_{i-1}(z)f_1(z)+ B_{i}(z)f_2(z)$ vanishes with high multiplicity at 0, for any $i\in\zeroN$. The point here is that $B_{-1}$ and $B_N$ are considered to be identically zero, so that for $i=N$ the function $A_N(z)+B_{N-1}(z)f_1(z) $ vanishes to high order at 0, and is therefore presumably small at $z=1$.

Let us come back to our general setting now.
 Recall that the parameters are given by Eqs.~\eqref{eqcst1} and~\eqref{eqcst2}; they are the same as in \cite{Zudilin}, except that $\ewa$ and $M$ are not fixed in terms of $N$. 
The following construction is exactly \cite[Lemma 1.1]{Zudilin}.

\begin{lem} \label{lemconstructionsiegel}
There exist polynomials $P_\capa(z)$ of degree less than $M$, for $\capa\in\Omega$, not all zero, such that 
$$ \ord_0 \Big( P_\capa(z) + \sum_{j=1}^m P_{\capa-e_j}(z)f_j(z)\Big) \geq K \quad \mbox{ for any } \capa\in\Theta$$
and
$$\pi_{\capa,\nu}\in\Z, \quad | \pi_{\capa,\nu} | \leq C_0^{\omega M/\ewa}$$
for any $\capa\in\Omega$ and any $\nu\in\zeroMmu$, 
where the coefficients $\pi_{\capa,\nu} $ are defined by
$$ P_\capa(z) = \sum_{\nu=0}^{M-1} \frac{\pi_{\capa,\nu}}{\nu!} z^\nu \quad \mbox{ for any } \capa\in\Omega.$$
\end{lem}

Here and below, we denote by $C_0$, $C_1$, \ldots, $C_4$ positive constants that depend only on $f_1$, \ldots, $f_m$ (except that $C_4$ depends also on $N$). 

\medskip

Now recall that all coefficients $S_{l,j}(z) $ of the differential system~\eqref{eqsysdiffinitial} belong to $\Q[z,1/z]$, i.e. that this system has no non-zero finite singularity. Therefore denoting by $T(z)$ the least common denominator of the $S_{l,j}(z) $, we have 
 \begin{equation}\label{eqTmonome}
 T(z)=\tau z^i \mbox{ for some $i\in \N$ and $\tau\in\N\etoile$, and } T(z)S_{l,j}(z)\in \Z[z] \mbox{ for any } l,j.
\end{equation}

As in \cite[Eq. (1.8)]{Zudilin} we define recursively polynomials $P^{[k]}_{\capa}(z)$, for $k\geq 1$ and $\capa\in\Omega$, by letting $P^{[1]}_{\capa}(z)= P_{ \capa}(z)$ and for any $k\geq 1$ and any $\capa\in\Omega$,
 \begin{eqnarray}
 P^{[k+1]}_{\capa}(z) &=& T(z)\Big(\frac{\dd}{\dd z} P^{[k]}_{\capa}(z) + (|\capa|+1-N) \sum_{l=1}^m S_{l,0}(z)P^{[k]}_{\capa-e_l}(z) \label{eqdefpkcapa} \\
&&\quad\quad \quad\quad\quad\quad- \sum_{l=1}^m \sum_{j=1}^m (\capa_j - \delta_{l,j}+1)S_{l,j}(z)P^{[k]}_{\capa-e_l+e_j}(z) \Big).
 \nonumber
 \end{eqnarray}
We recall that $\delta_{l,j}$ is Kronecker's symbol, and whenever $ \capa-e_l\in\Z^m$ (resp. $\capa-e_l+e_j$) has a negative coefficient, the corresponding term is omitted.
The only difference with \cite[Eq.~(1.8)]{Zudilin} is a shift in the index $k$: our $P^{[k]}_{\capa}(z)$ is denoted by $P^{[k-1]}_{\capa}(z)$ in \cite{Zudilin}.
The connection of this definition of $P^{[k]}_{\capa}(z)$ with a differential system will be explained in \S \ref{secsystemediff} below. 

The following result is part $(a)$ of \cite[Lemma 1.3]{Zudilin}, with coefficients $\pi_{k,\capa,\nu} $ defined by
$$ P^{[k]}_{\capa}(z) = \sum_{\nu=0}^{M+t(k-1)-1} \frac{\pi_{k,\capa,\nu}}{\nu!} z^\nu \mbox{ for any } \capa\in\Omega\mbox{ and any } k\geq 1,$$
and 
$$ t= \max\Big(\deg T(z), \max_{1\leq l \leq m} \max_{0\leq j \leq m}\deg(T(z)S_{l,j}(z))\Big).$$

\begin{lem} \label{lemzu2}
For any $\capa\in\Omega$, any $k\geq 1$ and any $\nu\in \zerobiz$, we have $\pi_{k,\capa,\nu}\in\Z$, and 
$$
| \pi_{k,\capa,\nu} | \leq C_0^{\omega M/\ewa} M^{C_2 \ewa M} \quad \mbox{ if } k < C_1 \ewa M.
$$
\end{lem}

The following result is the special case $l\etoile=1$ and $\alpha=1$ of part $(b)$ of \cite[Lemma 1.3]{Zudilin}. 

\begin{lem}\label{lemzu3} For any $k\geq 1$ such that $k < C_1 \ewa M$, we have
$$ 
\Big| P^{[k]}_{(N,0,0,\ldots,0)}(1) +P^{[k]}_{(N-1,0,0,\ldots,0)}(1) f_1(1) 
\Big|\le C_0^{\omega M/\ewa} M^{C_2 \ewa M} C_3^M M^{-K}.
$$
\end{lem}

The main point of the present paper is the following result. Under the assumption that $f_1$, \ldots, $f_m$ are algebraically independent, it 
 is proved in \cite[Lemma~3.5]{Zudilin} in a slightly weaker but explicit form, namely with $ \lfloor 2\ewa M \rfloor + \omega$ instead of $ \lfloor \ewa M \rfloor +C_4$. It is now optimal, up to the value of $C_4$ (and the assumption $\ewa\leq \frac1{\omega+1}$, which is harmless in our application).
 
 \begin{prop} \label{propmatriceinversible} If $\ewa\leq \frac1{\omega+1}$ then 
there exists a constant $C_4$, which depends on $f_1$, \ldots, $f_m$ and on $N$ (but not on $M$ or $\ewa$), such that the matrix $(P^{[k]}_{\capa}(1))_{\capa\in\Omega, 1\leq k \leq \lfloor \ewa M \rfloor +C_4}$ has rank $\omega$.
\end{prop}

\begin{Remark}\label{remeff2} We shall prove that $C_4$ can (in principle) be computed effectively in terms of $f_1$, \ldots, $f_m$ and $N$.
\end{Remark}

Proposition~\ref{propmatriceinversible} will be proved 
 in \S \ref{secpreuvematriceinversible}, using Theorem~\ref{thnondeg} proved in \S \ref{subsecnondeg}, the differential system considered in \S \ref{secsystemediff} and Theorem~\ref{thrat} proved in \S \ref{sectech}.
In the next two sections, we admit 
Proposition~\ref{propmatriceinversible} and deduce from it the results announced in the introduction.

\subsection{Proof of Theorem~\ref{theo:main} for $E$-functions in the strict sense} \label{proofthm1}

Let us now prove Theorem~\ref{theo:main} for $E$-functions in the strict sense, 
following \cite[pp. 581--583]{Zudilin} but without explicit expressions for $\ewa$ and $M$. 

Starting with an $E$-function $g \in\Q[[z]]$ and $r\in\Q$ such that $g(r)\not\in\Q$, we construct $f_1$, \ldots, $f_m$ as in \S \ref{secnotations} so that $f_1(1)=g(r)$.
Let $\eps>0$; we may assume that $\eps$ is sufficiently small in terms of $f_1$, \ldots, $f_m$. We choose $N = \lfloor m/\eps \rfloor +1 $ so that 
$$ N\geq m/\eps $$
and $N$ can be made sufficiently large in terms of $f_1$, \ldots, $f_m$. 
We recall that $\omega/\theta = 1+\frac{N}{N+m-1}$, so that 
$$ \Big(1+\frac{m}{N}\Big) \Big( 1-\frac{\omega}{\theta}\Big) < -1.$$
Using this bound and the fact that $C_2$, $t$, $\omega$ and $\theta$ depend only on $N$ and $f_1$, \ldots, $f_m$, we may choose $\ewa>0$ sufficiently small (with respect to $N$, $f_1$, \ldots, $f_m$) so that Eq.~\eqref{eqcst1} holds, $\ewa\leq 1/(\omega+1)$ (as assumed in Proposition~\ref{propmatriceinversible}), and
\begin{equation}\label{eqfinmu}
\Big(1+\frac{m}{N}\Big)\Big(1+t\ewa+C_2 \ewa - \frac{\omega-\ewa}{\theta}\Big)< - (1+t\ewa+C_2\ewa).
\end{equation}

\bigskip

In what follows, we assume $M $ to be sufficiently large in terms on $\ewa$, $N$, $f_1$, \ldots, $f_m$; we shall denote by $C_j(\ewa,N)$ positive constants that depend on $\ewa$, $N$, $f_1$, \ldots, $f_m$.

\bigskip

Since the matrix 
$(P^{[k]}_{\capa}(1))_{\capa\in\Omega, 1\leq k \leq \lfloor \ewa M \rfloor +C_4}$ of Proposition~\ref{propmatriceinversible} has rank $\omega$, the submatrix consisting in the columns indexed by $\capa=(N,0,\ldots,0)$ and $\capa=(N-1,0,\ldots,0)$ has rank~2. This provides positive integers $k_1,k_2< \ewa M +C_4$ such that 
\begin{equation}\label{eqfin0}
\det\left(\begin{matrix}
 P^{[k_1]}_{(N,0,0,\ldots,0)}(1) & P^{[k_1]}_{(N-1,0,0,\ldots,0)}(1) \\
 P^{[k_2]}_{(N,0,0,\ldots,0)}(1) & P^{[k_2]}_{(N-1,0,0,\ldots,0)}(1) 
\end{matrix}
\right) \neq 0.
\end{equation}
For any $j\in\undeux$, let 
$$
p_j = - (M+tk_j)!\, P^{[k_j]}_{(N,0,0,\ldots,0)}(1) 
\quad \mbox{ and }\quad 
q_j = (M+tk_j)!\, P^{[k_j]}_{(N-1,0,0,\ldots,0)}(1) .
$$
Lemma~\ref{lemzu2} yields $p_j,q_j\in\Z$, and also
\begin{eqnarray}
 |q_j|
 &\leq & (M+t(\ewa M+C_4))!\, e C_0^{\omega M/\ewa} M^{C_2\ewa M} \nonumber \\
 &\leq & (M(1+ t\ewa) +tC_4))^{tC_4} \, (M(1+ t \ewa))!\, e C_0^{\omega M/\ewa} M^{C_2\ewa M}\nonumber \\
 &\leq & C_5(\ewa,N)^M \, M^{(1+t\ewa+C_2\ewa)M}
 \label{eqfin1}
\end{eqnarray}
provided $M\geq C_6(\ewa,N) $, and Lemma~\ref{lemzu3} yields 
in the same way (using Eq.~\eqref{eqcst2})
\begin{eqnarray}
 |q_j f_1(1)-p_j|
 &\leq & (M+t(\ewa M+C_4))!\, C_0^{\omega M/\ewa} M^{C_2\ewa M} C_3^M M^{-K} \nonumber \\
 &\leq & C_7(\ewa,N)^M \, M^{(1+t\ewa+C_2\ewa-\frac{\omega-\ewa}{\theta})M}
 \label{eqfin2}
\end{eqnarray}
if $M\geq C_8(\ewa,N) $.

\bigskip

Now let $p\in\Z$ and $q\in \N\etoile$; upon changing the constant $c $ in Theorem~\ref{theo:main}
(since $f_1(1)\not\in\Q$), we may assume that $|p| $ and $q$ are sufficiently large (with respect to $\ewa$, $N$, $f_1$, \ldots, $f_m$, since these quantities have been chosen in terms of $f_1(1)$ and $\eps$ only). We choose for $M$ the least integer such that 
\begin{equation}\label{eqfin3}
C_7(\ewa,N)^M \, M^{(1+t\ewa+C_2\ewa-\frac{\omega-\ewa}{\theta})M}\leq \frac{1}{2q}. 
\end{equation}
This integer exists because we have assumed $\ewa>0 $ sufficiently small in terms of $N$, $f_1$, \ldots, $f_m$, so that $ 1+t\ewa+C_2\ewa-\frac{\omega-\ewa}{\theta} < 0$; moreover $M$ can be made large enough (in terms of $\ewa$, $N$, $f_1$, \ldots, $f_m$) by assuming that $q$ is. Then Eq.~\eqref{eqfin2} yields
$$ q \, | q_j f_1(1)-p_j|\leq 1/2 \mbox{ for any } j\in\undeux.$$
Now Eq.~\eqref{eqfin0} yields 
$\det\left( \begin{matrix}
 p_1 & q_1 \\ p_2 & q_2
\end{matrix}
\right)\neq 0$, so 
that $(p,q)$ is non-collinear to at least one of the $(p_j,q_j)$, $j\in\undeux$. For this index $j$ we have
\begin{equation}\label{eq:detpjqj}
\det\left( \begin{matrix}
 p & p_j \\ q & q_j
\end{matrix}
\right)\in\Z\setminus\{0\}.
\end{equation}
This determinant is also equal to 
$$
\det\left( \begin{matrix}
 p- qf_1(1) & p_j-q_j f_1(1) \\ q & q_j
\end{matrix}
\right),
$$
so that 
\begin{align*}
 |q_j| \, |q f_1(1)-p |
 &= \Big| \det\left (\begin{matrix}
 p & p_j \\ q & q_j
\end{matrix}
\right) - q \, ( q_j f_1(1) - p_j)\Big| \\
&\geq 1 - q \, | q_j f_1(1)-p_j| \geq 1-1/2 = 1/2
\end{align*}
and, using Eqs.~\eqref{eqfin1} and~\eqref{eqfinmu}:
\begin{align*}
 |q f_1(1)-p |&\geq \frac{1}{2|q_j|}
 \geq \frac12 C_5(\ewa,N)^{-M} \, M^{-(1+t\ewa+C_2\ewa)M}\\
 &\geq\Big[ C_7(\ewa,N)^{M-1} \, (M-1)^{(1+t\ewa+C_2\ewa-\frac{\omega-\ewa}{\theta})(M-1)}\Big]^{1+m/N}
\end{align*}
provided $M\geq C_9(\ewa,N)$. Since $M$ is the least integer such that 
Eq.~\eqref{eqfin3} holds, we deduce that 
$$ 
 |q f_1(1)-p | > \Big( \frac1{2q}\Big)^{1+m/N}.
$$
Since $m/N\leq \eps$ this concludes the proof of Theorem~\ref{theo:main} for $E$-functions in the strict sense.

\subsection{Proof of Theorem~\ref{theo:main} for $E$-functions in Siegel's sense} \label{genethm1}
In this section, we explain the changes that must be done to obtain Theorem~\ref{theo:main} for any $E$-function $f$ in Siegel's original sense.

$\bullet$ Firstly, in the proof of Theorem~\ref{theo:main} for $E$-functions in the strict sense, we use various results of Andr\'e and Beukers that they have proved only for $E$-functions in the strict sense (using the theory of $E$-operators due to the former). Since then, all these results have been proved to hold {\em verbatim} for $E$-functions in Siegel's sense by Lepetit~\cite{lepetit}, completing the results already given in \cite[pp.~746--747]{andre2}.

\smallskip

$\bullet$ Secondly, in \S \ref{construction} we use {\em verbatim} Zudilin's estimates that he has also proved only for $E$-functions in $\mathbb Q[[z]]$ in the strict sense. Let us mention the changes that must be made to his lemmas to deal with Siegel's $E$-functions in $\mathbb Q[[z]]$. We recall that the archimedean and non-archimedean bounds on the Taylor coefficients of Siegel's $E$-functions are of the form ``for all $\eps'>0$, $ \ldots \le n^{\eps'n}$ for all $n\ge N(\eps')$''. In 
Lemma~\ref{lemconstructionsiegel}, this changes the quantity $C_0^{\omega M/\ewa}$ by $M^{\omega \eps'M/\ewa}$, where $\eps'>0$ is fixed and independent of the other parameters but arbitrarily small, and $M\ge M_0(\eps')$. The same remark applies in Lemma~\ref{lemzu2}, where $C_0^{\omega M/\ewa} M^{C_2 \ewa M} $ becomes $M^{\omega \varepsilon'M/\eta+ C_2 \ewa M} $, and in Lemma~\ref{lemzu3}, where $ M^{-K}$ reads $M^{-K(1-\eps')}$ and the constant $C_3$ is also possibly changed but it still does not depend on $M$. With these estimates, we conclude the proof as that of Theorem~\ref{theo:main} for $E$-functions in the strict sense because $\eps'$ can be taken arbitrarily small provided $M$ is assumed to be large enough, which can be assumed as in \S\ref{proofthm1}. 

\medskip

The rest of the present paper is devoted to a proof of Proposition~\ref{propmatriceinversible}, which has been admitted in \S \S \ref{proofthm1} and~\ref{genethm1}.

\subsection{Differential system}\label{secsystemediff}

In this section we define a matrix $A\in M_\omega(\Q(z))$ and consider the differential system $Y'=AY$, of which solutions will be constructed in Proposition~\ref{propycapa}. As stated in \S \ref{secnotations} a bijective map $\unomega\to\Omega$ is fixed, so that a solution $Y$ is a vector $(y_\capa(z))$ indexed by $\capa\in\Omega$. Here and below, we identify tuples in $\C^q$ with column matrices in $M_{q,1}(\C)$.

We shall also relate the notation $P^{[k]}_{\capa}$ of \S \ref{construction} to the $P_{k,\capa}$ of \S \ref{seclemzeros}; in what follows we will use mostly the notation of \S \ref{seclemzeros}, including
\begin{equation}\label{eqdefR}
 R(Y)(z)=\sum_{\capa\in\Omega} P_\capa(z)y_\capa(z)
\end{equation}
when $Y =(y_\capa(z))$ is a solution of the differential system $Y'=AY$.

\bigskip

The matrix $A = (A_{\lambda,\capa}(z))_{\lambda,\capa\in\Omega}\in M_\omega(\Q(z))$ that we consider is defined (in terms of the coefficients $S_{l,j}(z)$ of the differential system~\eqref{eqsysdiffinitial}) 
by 
\begin{equation}\label{eqdefA}
 A_{\lambda,\capa}(z) = 
 \begin{cases}
 -\lambda_j S_{l,j}(z) \;\mbox{ if $\capa=\lambda-e_j+e_l$ for some $j,l\in\unm$ with $j\neq l$,}\\
 -\sum_{j=1}^m \lambda_j S_{j,j}(z) \; \mbox{ if $\capa=\lambda$,}\\
 S_{j,0}(z) \; \mbox{ if $\capa=\lambda+e_j$ and $|\lambda| = N-1$,}\\
 0 \;\mbox{ otherwise,}
 \end{cases}
\end{equation}
as in \cite[Eq.~(3.2)]{Zudilin}. 
We recall from \S \ref{secnotations} that all rational functions $S_{l,j}(z)$, and therefore all coefficients of $A$, belong to $\Q[z,1/z]$.
With this definition, Eq.~\eqref{eqdefpkcapa} reads
$$
\left( \begin{array}{c} P^{[k+1]}_{1}(z) \\ \vdots \\ P^{[k+1]}_{\omega}(z)
\end{array}\right) = T(z) \left(\frac{\dd}{\dd z} + \, \tra A(z)\right) 
 \left( \begin{array}{c} P^{[k]}_{1}(z) \\ \vdots \\ P^{[k]}_{\omega}(z) \end{array} \right).
 $$ 
Except for the multiplicative factor $T(z)$ (used to ensure that all $P^{[k ]}_{\capa}(z)$ are polynomials), this is the same recurrence relation as the one used in \S \ref{subseclemzeroquebec} to define the rational functions $P_{k,\capa}(z)$ (see Eq.~\eqref{eqdefpki}); notice that $P_{k,\capa}(z)\in \Q[z,1/z]$ since $A \in M_\omega(\Q[z,1/z])$. Using the fact that $P_{1,\capa} = P^{[1]}_{\capa} = P_\capa$ by definition, we obtain by induction that
$$
\left( \begin{array}{c} P^{[k]}_{1}(z) \\ \vdots \\ P^{[k]}_{\omega}(z)
\end{array}\right) = T(z)^{k-1}
\left( \begin{array}{c} P_{k,1}(z) \\ \vdots \\ P_{k,\omega}(z)
\end{array}\right) + \sum_{k'=1}^{k-1} U_{k,k'}(z) \left( \begin{array}{c} P_{k',1}(z) \\ \vdots \\ P_{k',\omega}(z)
\end{array}\right)
$$
with rational functions $U_{k,k'}(z) \in\Q[z,1/z]$, since $T(z)$ and all coefficients of $A(z)$ belong to $ \Q[z,1/z]$. Now recall from Eq.~\eqref{eqTmonome} that $T(z)=\tau z^i$ for some $i\in\N$, so that $T(1) = \tau\in\N\etoile$ and 
\begin{equation} \label{eqequivlemzeros}
\rk ( P^{[k]}_{\capa}(1))_{\capa\in\Omega, 1\leq k \leq k_0} = \rk ( P_{k,\capa}(1))_{\capa\in\Omega, 1\leq k \leq k_0}
\end{equation}
for any $k_0\geq 1$. This equality will be used at the end of \S \ref{secpreuvematriceinversible}
to prove Proposition~\ref{propmatriceinversible}, since Theorem~\ref{thnondeg} (stated and proved in \S \ref{subsecnondeg}) yields a lower bound on $\rk ( P_{k,\capa}(1))_{\capa\in\Omega, 1\leq k \leq k_0}$. 

\bigskip

The end of this section is devoted to the proof of the following result; notice that parts $(i)$ and $(ii)$ are essentially 
proved in \cite[pp. 575--576]{Zudilin}. For $\capa\in\Theta$ we define $ Z_\capa = (z_{\capa,\lambda})_{\lambda\in\Omega}\in\C^\omega$ by 
$$
z_{\capa,\lambda} = \left\{
\begin{array}{l}
 1 \mbox{ if $\lambda=\capa$,} \\
 f_j(1) \mbox{ if $\lambda=\capa-e_j$ for some $j\in\unm$,}\\
 0 \mbox{ otherwise.}
\end{array}
\right.
$$

\begin{prop} \label{propycapa}
 There exist solutions $Y_\capa(z)$ of the differential system $Y'=AY$, for $\capa\in\Theta$, such that:
 \begin{itemize}
 \item[$(i)$] For any $\capa\in\Theta$, $Y_\capa(1)=Z_\capa$.
 \item[$(ii)$] The functions $Y_\capa(z)$, $\capa\in\Theta$, are linearly independent over $\C$.
 \item[$(iii)$] For any $\capa\in\Theta$, we have $R(Y_\capa)(z)=O(z^{K-cN})$ as $z\to0$, where $c>0$ is a constant that depends only on $f_1$, \ldots, $f_m$.
 \item[$(iv)$] For any $\capa\in\Theta$, the function $R(Y_\capa)(z)$ belongs to the Nilsson class at 0.
 \end{itemize}
\end{prop}

\begin{Remark}\label{remeff3}
The constant $c$ in part $(iii)$ can be made effective, using the results of \cite{BCY}.
\end{Remark}
\begin{proof} Consider the differential system
\begin{equation}\label{eqsysdiffnv}
 a'_k(z) = -S_{1,k}(z)a_1(z)- \ldots - S_{m,k}(z)a_m(z)
\mbox{ for any } k\in\unm.
\end{equation}
Since the system~\eqref{eqsysdiffinitial} of \S \ref{secnotations} has no non-zero finite singularity, all rational functions $S_{\ell,k}(z)$ belong to $\Q[z,1/z]$ and the system~\eqref{eqsysdiffnv} has no non-zero finite singularity. In particular, there exists a fundamental matrix of solutions $(\varphi_{k,l}(z))_{1\leq k,l\leq m}$ such that $\varphi_{k,l}(1)$ is equal to the Kronecker symbol $\delta_{k,l}$. Let $\vr_1,\ldots,\vr_m$ be independent variables, and put $a_k(z) = \sum_{l=1}^m \vr_l \varphi_{k,l}(z)$ for $k\in\unm$. Then $ (a_1(z),\ldots,a_m(z))$ is a solution of the system~\eqref{eqsysdiffnv}.

Consider the vector $\overline Y(z) = (\overline y_\lambda(z))_{\lambda\in\Omega}$ defined by:
\begin{equation} \label{eqdefylb}
\overline y_\lambda(z) =
\begin{cases}
 a_1(z)^{\lambda_1} \ldots a_m(z)^{\lambda_m} \; \mbox{ if } | \lambda | = N, \\
 a_1(z)^{\lambda_1} \ldots a_m(z)^{\lambda_m} ( 1+ a_1(z) f_1(z) + \ldots + a_m(z) f_m(z))\;\mbox{ if } | \lambda | = N-1.
\end{cases}
\end{equation}
Each of these functions is a polynomial in the variables $\vr_1,\ldots,\vr_m$, with coefficients that depend on $z$; all monomials that appear in this expression have total degree $N-1$ or $N$. Therefore we have
\begin{equation}\label{eqdefycapa}
 \overline Y(z) = \sum_{\capa\in\Omega} \vr_1^{\capa_1} \ldots \vr_m^{\capa_m} Y_\capa(z)
\end{equation}
and this expression defines functions $Y_\capa(z) $ independent from $\vr_1,\ldots,\vr_m$. We shall be interested in these functions only when $\capa\in\Theta$, i.e. $|\capa|=N$.

\bigskip

To prove part $(i)$, we deduce from $\varphi_{k,l}(1)=\delta_{k,l}$ that $ a_k(1)=\vr_k$, and Eq.~\eqref{eqdefylb} yields
\begin{equation}\label{eqcalculybar}
 \overline y_\lambda(1) =
\begin{cases}
 \vr_1^{\lambda_1} \ldots \vr_m^{\lambda_m} \; \mbox{ if } | \lambda | = N, \\
 \vr_1^{\lambda_1} \ldots \vr_m^{\lambda_m} ( 1+ \vr_1 f_1(1)+\ldots+ \vr_m f_m(1)) \;\mbox{ if } | \lambda | = N-1.
\end{cases}
\end{equation}
Given $\capa\in\Theta$, we write $Y_\capa(1) = (z_{\capa,\lambda})_{\lambda\in\Omega}$. Then Eq.~\eqref{eqdefycapa} shows that $z_{\capa,\lambda}$ is the coefficient of $ \vr_1^{\capa_1} \ldots \vr_m^{\capa_m}$ in the expression of $ \overline y_\lambda(1) $. Using Eq.~\eqref{eqcalculybar}, we obtain
$$ z_{\capa,\lambda} = 
\begin{cases}
 1 \; \mbox{ if } \lambda=\capa, \\
 f_j(1) \; \mbox{ if } \lambda=\capa - e_j \mbox{ for some } j\in\unm,\\
 0 \; \mbox{ otherwise. }
\end{cases}
$$
By definition of $Z_\capa$, this means $Y_\capa(1) =Z_\capa $ and concludes the proof of part $(i)$.

Part $(ii)$ follows easily from part $(i)$. Indeed we consider the matrix $M\in M_{\omega,\theta}(\C)$ with columns $Z_\capa$, $\capa\in\Theta$. We may assume that the bijective map $\unomega\to \Omega$ we have chosen in \S \ref{secnotations} maps $\untheta$ to $\Theta$: it allows us to identify $\Theta$ and $\untheta$. Then by definition on the $Z_\capa$, we have $M = \left(\begin{array}{c}I \\ M' \end{array}\right)$ for some matrix $M'\in M_{\omega-\theta,\theta}(\C)$, where $I\in M_{ \theta}(\C)$ is the identity matrix. Therefore $M$ has rank $\theta$, and the vectors $Z_\capa$ are linearly independent over $\C$. Using part $(i)$, this concludes the proof of $(ii)$.

Let us prove parts $(iii)$ and $(iv)$ now. For brevity we let $\rho^\capa=\rho_1^{\capa_1}\ldots\rho_m^{\capa_m} $ and define $a(z)^\lambda$ in an analogous way. We have:
\begin{align*}
 \sum_{\capa\in\Omega}\rho^\capa R(Y_\capa)(z)
 &= R(\overline Y)(z) = \sum_{\lambda\in\Omega} P_\lambda(z)\overline y_\lambda(z)\mbox{ using Eqns.~\eqref{eqdefR} and~\eqref{eqdefycapa}}\\
 &= \sum_{\lambda\in\Theta} P_\lambda(z) a(z)^\lambda + \sum_{\lambda\in\Omega\setminus \Theta} P_\lambda(z) a(z)^\lambda \Big(1+\sum_{j=1}^m a_j(z) f_j(z)\Big) \;\mbox{ by Eq.~\eqref{eqdefylb}}\\
 &= \sum_{\lambda\in\Theta} a(z)^\lambda \Big(P_\lambda(z) +\sum_{j=1}^m P_{\lambda-e_j}(z) f_j(z)\Big) + \sum_{\lambda\in\Omega\setminus \Theta} P_\lambda(z) a(z)^\lambda. 
\end{align*}
Now recall that $a_k(z) = \sum_{l=1}^m \vr_l \varphi_{k,l}(z)$. In the previous expression, we fix $\capa\in\Theta$ and identify the coefficients of $\rho^\capa$ in both sides. Since the second term of the right hand side is homogeneous of degree $N-1$, whereas $|\capa|=N$, it does not contribute and we have
\begin{equation}\label{equtilepropycapa}
R(Y_\capa)(z) = \sum_{\lambda\in\Theta} b_{\lambda,\capa}(z) \Big(P_\lambda(z) +\sum_{j=1}^m P_{\lambda-e_j}(z) f_j(z)\Big) 
\end{equation}
where $ b_{\lambda,\capa}(z) $ is the coefficient of $\rho^\capa$ in the expansion of $ a(z)^\lambda $. This coefficient $ b_{\lambda,\capa}(z) $ is an explicit homogeneous polynomial of degree $N$ (with constant integer coefficients) in the functions $\varphi_{k,l}(z)$. 

Now denote by $Z'=SZ$ the differential system~\eqref{eqsysdiffinitial} of \S \ref{secnotations}. It has 
 at worst a regular singularity at 0, and therefore admits a fundamental matrix of solutions $M(z)$ with coefficients 
 in the Nilsson class at 0. Then $\tra M(z)^{-1}$ is a fundamental matrix of solutions of the (dual) differential system $Y'=-\, \tra SY$. Removing the first coordinate of the solutions yields a fundamental matrix of solutions of the differential system~\eqref{eqsysdiffnv}, all of which coefficients 
are in the Nilsson class at 0. Accordingly all $\varphi_{k,l}$ and all $ b_{\lambda,\capa} $ belong to the Nilsson class at 0, and so does $R(Y_\capa)$ using Eq.~\eqref{equtilepropycapa}: this proves part $(iv)$. Moreover, 
 there exists a constant $c$, which depends only on this system, such that $\varphi_{k,l}(z) = O(z^{-c})$ as $z\to 0$. Therefore we have $ b_{\lambda,\capa}(z)= O(z^{-cN})$ for any $\lambda,\capa\in\Theta $. Using Eq.~\eqref{equtilepropycapa} and Lemma~\ref{lemconstructionsiegel} this concludes the proof of part $(iii)$.

\medskip

To conclude the proof of Proposition~\ref{propycapa}, let us prove that $Y'_\capa=AY_\capa$ for any $\capa\in\Theta$. Since $A$ does not depend on the $\rho_j$, it is enough using Eq.~\eqref{eqdefycapa} to prove that $\overline Y' = A\overline Y$.
To begin with, Eq.~\eqref{eqdefylb} yields for any $\lambda\in\Theta$:
\begin{align*}
 \overline y_\lambda'(z) 
 &= - \sum_{j=1}^m \lambda_j a(z)^{\lambda - e_j} \sum_{l=1}^m S_{l,j}(z)a_l(z) \; \mbox{ using Eq.~\eqref{eqsysdiffnv}}\\
 &= - \sum_{j=1}^m \sum_{l=1}^m \lambda_j S_{l,j}(z) a(z)^{\lambda - e_j+e_l}\\
 &= \sum_{\capa\in \Omega} A_{\lambda,\capa}(z) \overline y_\capa(z) \; \mbox{ using Eq.~\eqref{eqdefA}.}
\end{align*}
To prove the same formula for $\lambda\in\Omega\setminus \Theta$, we notice that, using Eqns.~\eqref{eqsysdiffnv} and~\eqref{eqsysdiffinitial}:
\begin{align*}
 \frac{\dd}{\dd z} \Big( 1+&\sum_{p=1}^m a_p(z)f_p(z)\Big) 
 = \sum_{p=1}^m a'_p(z)f_p(z) + \sum_{l=1}^m a_l(z)f'_l(z)\\
 &= - \sum_{p=1}^m \sum_{l=1}^m S_{l,p}(z) a_l(z)f_p(z) + \sum_{l=1}^m a_l(z) ( S_{l,0}(z)+ \sum_{p=1}^m S_{l,p}(z)f_p(z))\\
 &= \sum_{l=1}^m S_{l,0}(z) a_l(z) .
\end{align*}
This gives for any $\lambda\in\Omega\setminus \Theta$, using Eqns.~\eqref{eqdefylb} and~\eqref{eqsysdiffnv}:
\begin{align*}
 \overline y_\lambda'(z) 
 &= a(z)^{\lambda } \sum_{p=1}^m S_{p,0}(z) a_p(z) + \sum_{j=1}^m \lambda_j a(z)^{\lambda - e_j} a'_j(z) ( 1+\sum_{p=1}^m a_p(z)f_p(z) ) \\
 &= \sum_{p=1}^m S_{p,0}(z) a(z)^{\lambda + e_p } - \sum_{j=1}^m \lambda_j a(z)^{\lambda -e_j } \sum_{l=1}^m S_{l,j}(z)a_l(z) ( 1+\sum_{p=1}^m a_p(z)f_p(z) ) \\
 &= \sum_{p=1}^m S_{p,0}(z) \overline y_{\lambda+e_p} (z) - \sum_{j=1}^m \sum_{l=1}^m \lambda_j 
 S_{l,j}(z) \overline y_{\lambda-e_j+e_l} (z) \\
 &=\sum_{\capa\in \Omega} A_{\lambda,\capa}(z) \overline y_\capa(z) \; \mbox{ using Eq.~\eqref{eqdefA}.}
\end{align*}
This concludes the proof of Proposition~\ref{propycapa}.
\end{proof}

\subsection{Application of Shidlovskii's lemma: proof of Proposition~\ref{propmatriceinversible}} \label{secpreuvematriceinversible}

In this section we prove Proposition~\ref{propmatriceinversible} stated in \S \ref{construction}, as consequence of Theorem~\ref{thnondeg} proved in \S \ref{subsecnondeg} (using also Theorem~\ref{thrat} that will be stated and proved in \S \ref{sectech}). We apply Theorem~\ref{thnondeg} in the setting of \S \ref{secsystemediff}, namely with $q=\omega$ and the differential system $Y'=AY$ where $A$ is defined by Eq.~\eqref{eqdefA}. All vector spaces, dimensions and other notions of linear algebra are over $\C$. 

We let $\calF = \Span\{Y_\capa, \, \capa\in\Theta\}$ where the solutions $Y_\capa$ are given by Proposition~\ref{propycapa}. They are in the Nilsson class at 0, and are linearly independent over $\C$ so that $\dim\calF = \theta$. Given {\em any} family of polynomials $P_\capa(z)\in\Qbar[z]$, for $\capa\in\Omega$, we denote by $\calRP $ the space of solutions $Y$ of $Y'=AY$ such that $R(Y)(z)=\sum_{\capa\in\Omega}P_\capa(z)y_\capa(z)$ is identically zero.

\medskip

We consider the linear map $\evun: Y \mapsto Y(1)$, from the space of solutions of $Y'=AY$ to $\C^\omega$. It is bijective because 1 is not a singularity of this differential system, defined by~\eqref{eqdefA} with $S_{l,j}(z) \in\Q[z,1/z]$ for any $l,j$. We let $F=\evun(\calF)$ and $R=\evun(\calRP)$. The space $F$ is spanned by the vectors $Z_\capa$, $\capa\in\Theta$, by Proposition~\ref{propycapa} $(i)$. The space $R$ is defined over $\Qbar$ using Lemma~\ref{lemrrationnel}.

By construction (see \S \ref{secnotations}) the functions 1, $f_1(z)$, \ldots, $f_m(z)$ are linearly independent over $\C(z)$, and they make up a vector solution of a differential system of order 1 with coefficients in $\Q[z,1/z]$, given by Eq.~\eqref{eqsysdiffinitial}. Since $1$ is not a singularity of this system, Beukers' refinement \cite[Corollary 1.4]{beukers} of the Siegel-Shidlovskii theorem 
shows that the values at 1 of these functions are linearly independent over $\Qbar$. Therefore Theorem~\ref{thrat} applies with $\xi_1=f_1(1)$, \ldots, $\xi_m=f_m(1)$. It shows that if $R$ is proper, then 
$$\dim R > \Big( 2 - \frac{m-1}{N+m-1}\Big) \dim(F\cap R),$$
so that 
$$
\frac{\dim(\calF\cap\calRP)}{\dim \calF} < \frac{\dim \calRP }{q}
$$
since $\dim\calF = \theta$ and $ q= \omega$. Therefore $\calF$ is nondegenerate over $\Qbar$ (as defined in \S \ref{subsecnondeg}). 

\medskip

Now we choose $ n=M-1$ and consider the polynomials $P_\capa(z)$, $\capa\in\Omega$, defined in Lemma~\ref{lemconstructionsiegel}.
They belong to $\Qbar[z]$, have degree at most $n$, are not all zero, and for any $Y \in\calF$ the remainder $R(Y)(z) $ belongs 
 to the Nilsson class at $0$ and vanishes at $0$ with order at least $K- c N = \lfloor \frac{(q-\ewa)(n+1)}{\dim\calF}\rfloor - cN \geq \frac{(q-\ewa)n}{\dim\calF} - w $ with $w = cN+1$, using Proposition~\ref{propycapa}. 
 Assuming $\ewa \leq 1/(q+1)$ and $n$ large enough (in terms of $f_1, \ldots, f_m$ and $N$),
Theorem~\ref{thnondeg} yields $\rk (P_{k,\capa}(1))_{\capa\in\Omega, 1\leq k \leq \lfloor \ewa M \rfloor+C_4} = \omega$ with $C_4 = \cstun+(cN+1)q$. Using Eq.~\eqref{eqequivlemzeros} proved in \S \ref{secsystemediff},
this concludes the proof of Proposition~\ref{propmatriceinversible}.

\section{Proof of non-degeneracy}\label{sectech}

In this section we state and prove Theorem~\ref{thrat}, a key ingredient in the proof of Proposition~\ref{propmatriceinversible} given in \S \ref{secpreuvematriceinversible}. The application of this result is explained in \S \ref{secpreuvematriceinversible}: it allows us to prove that the subspace $\calF$ we are interested in is non-degenerate over $\Qbar$, as defined in \S \ref{subsecnondeg}. This is a crucial assumption in our multiplicity estimate (namely Theorem~\ref{thnondeg} proved in \S \ref{subsecnondeg}). 

\medskip

Theorem~\ref{thrat} is an independent result, for which we need only the following notation. 
We consider integers $m,N\geq 1$ and complex numbers $\xi_1,\ldots,\xi_m$; we assume that $1$, $\xi_1$, \ldots, $\xi_m$ are linearly independent over $\Qbar$. As in \S \ref{secnotations} we let $\Omega = \{ \capa\in\N^m, \, \ineg\}$ where $| \capa | = \capa_1+\ldots+\capa_m$, and $\omega=\Card\, \Omega = \binom{N+m-2}{m-1} + \binom{N+m-1}{m-1}$; we denote by $(e_j)_{1\leq j \leq m}$ the canonical basis of $\Z^m$, and by $(E_\capa)_{\capa\in\Omega}$ that of $\C^\Omega$. In other words, we have $E_\capa = (\delta_{\capa,\capa'})_{\capa'\in\Omega}$ where $ \delta_{\capa,\capa'}$ is the Kronecker's symbol.
We consider 
$$ Z_\capa = E_\capa +\sum_{j=1}^m \xi_j E_{\capa - e_j} \mbox{ for any }\capa\in\Theta = \{ \capa\in\N^m, \, | \capa | = N \}$$
with the convention that $ E_{\capa - e_j} = 0$ if $\capa - e_j\in\Z^m$ has at least one negative component (namely, if $\capa_j=0$); with $\xi_j=f_j(1)$ these are exactly the vectors $Z_\capa$ defined before Proposition~\ref{propycapa} in \S \ref{secsystemediff}. We denote by $F$ the subspace of $\C^\Omega$ generated by these vectors, namely 
$$F = \Span\{Z_\capa, \, \capa\in\Theta\}.$$
It is not difficult (see the proof of $(ii)$ in Proposition~\ref{propycapa} above) to see that the $ Z_\capa$ are linearly independent, so that $\dim F = \binom{N+m-1}{m-1}$, denoted by $ \theta$.

\begin{theo} \label{thrat} 
Let $R$ be a vector subspace of $\C^\Omega$ defined over $\Qbar$. Then we have:
$$\dim R \geq \Big( 2 - \frac{m-1}{N+m-1}\Big) \dim(F\cap R),$$
and equality holds if and only if $R=\{0\}$ or $R=\C^\Omega$.
\end{theo}

By {\em defined over $\Qbar$}, or {\em rational over $\Qbar$}, we mean that $R$ has a $\mathbb C$-basis consisting in vectors in $\Qbar^\Omega$. This is equivalent to the existence of a system of linear equations with coefficients in $\Qbar$ that defines $R$. We point out that an inequality on dimensions, such as the one of Theorem~\ref{thrat}, is reminiscent of the notion of $(e,j)$-irrationality introduced in \cite{Joseph1} and \cite{Joseph2}.

\subsection{Two lemmas} \label{ssec:twolemmas}

The data of Theorem~\ref{thrat} depend only on $m,N\geq 1$ and on $\xi = ( \xi_1,\ldots,\xi_m)\in\C^m$; in the proof we shall often deduce Theorem~\ref{thrat} for some triples $(m,N,\xi)$ from the same statement for other triples. It will always be assumed, implicitly or explicitly, that $1$, $\xi_1$, \ldots, $\xi_m$ are linearly independent over $\Qbar$. We first prove two lemmas; recall that we identify tuples in $\C^m$ to column matrices in $M_{m,1}(\C)$. 

\begin{lem}\label{lemrat1} Let $m,N\geq 1$ and $\xi = ( \xi_1,\ldots,\xi_m)\in\C^m$, with $1$, $\xi_1$, \ldots, $\xi_m$ linearly independent over $\Qbar$. Let $A\in\GL_m(\Qbar)$; define $\xi' = ( \xi'_1,\ldots,\xi'_m)$ by $ \xi' = A\, \xi$.

If Theorem~\ref{thrat} holds for $(m,N,\xi)$ with any subspace $R$ defined over $\Qbar$ of a given dimension $\vr$, then it also does for $(m,N,\xi')$.
\end{lem}
\begin{proof} Decomposing $A$ into a product of simpler matrices, we may restrict to the following 3 cases.

$\bullet$ Case 1: $\xi'_j = \xi_{\sigma(j)}$ for any $j\in\unm$, with $\sigma\in \Sm$. In this case Lemma~\ref{lemrat1} is obvious, by permuting the coordinates in $\N^m$ and accordingly in $\C^\Omega$.

$\bullet$ Case 2: $\xi'_{j_0} =\lambda \xi_{j_0}$ for some $\lambda\in\Qbar\etoile$ and $j_0\in\unm$, and $\xi'_j=\xi_j$ for any $j\in\unm\setminus\{j_0\}$. Let $f:\C^\Omega\to\C^\Omega$ be the (bijective) linear map defined by $f( E_\capa ) = \lambda^{\capa_{j_0}} E_\capa$ for any $\capa\in\Omega$. We denote by $(Z_\capa)$ and $F$ (resp. $(Z'_\capa)$ and $F'$) the data associated with $\xi$ (resp. $\xi'$). Then we have for any $\capa\in\Theta$:
\begin{multline*}
f(Z'_\capa)=f(E_\capa)+\sum_{j=1}^m \xi'_j f(E_{\capa - e_j}) 
\\
= \lambda^{\capa_{j_0}} E_\capa + \lambda \xi_{j_0} \lambda^{\capa_{j_0}-1} E_{\capa - e_{j_0}} + \sum_{j\neq j_0} \xi_j \lambda^{\capa_{j_0}} E_{\capa - e_j} = \lambda^{\capa_{j_0}} f(Z_\capa)
\end{multline*}
so that $f(F')=F$. Let $R'$ be a subspace of $\C^\Omega$ defined over $\Qbar$, with $\dim R'=\vr$. Taking $R=f(R')$ we have $\dim R =\dim R'$, $\dim (F\cap R) =\dim (F'\cap R')$ and $R$ is defined over $\Qbar$. Therefore Theorem~\ref{thrat} applied to $(m,N,\xi)$ with $R$ shows that Theorem~\ref{thrat} holds for $(m,N,\xi')$ with $R'$.

$\bullet$ Case 3: $\xi'_{j_0} = \xi_{j_0} + \xi_{j_1}$ with distinct $j_0,j_1\in\unm$, and $\xi'_j=\xi_j$ for any $j\in\unm\setminus\{j_0\}$. We consider the linear map $f:\C^\Omega\to\C^\Omega$ given by
$$ f(E_\capa) = \sum_{t=0} ^{ \capa_{j_0} } \binom{t+ \capa_{j_1}}{ \capa_{j_1}} E_{\capa-t e_{j_0} + t e_{j_1}} \mbox{ for any } \capa\in\Omega.
$$ 
It is bijective because $f(E_\capa)-E_\capa$ is a linear combination of the $E_{\capa'}$ with $\capa'_{j_0}<\capa_{j_0}$. For any $\capa\in\Theta$ we have:
\begin{align}
f( &E_{\capa - e_{j_1}} ) + f( E_{\capa - e_{j_0}} )
 = \sum_{t=0} ^{ \capa_{j_0} } \binom{t+ \capa_{j_1}-1}{ \capa_{j_1}-1} E_{\capa-t e_{j_0} +( t-1) e_{j_1}} 
 + \sum_{t=0} ^{ \capa_{j_0} -1} \binom{t+ \capa_{j_1}}{ \capa_{j_1}} E_{\capa-(t+1) e_{j_0} + t e_{j_1}} \nonumber\\
 &= \sum_{t=0} ^{ \capa_{j_0} } \bigg( \binom{t+ \capa_{j_1}-1}{ \capa_{j_1}-1} + \binom{t+ \capa_{j_1}-1}{ \capa_{j_1}} \bigg)
 E_{\capa-t e_{j_0} +( t-1) e_{j_1}} \;\mbox{($t'=t+1$ in the second sum)} \nonumber\\
 &= \sum_{t=0} ^{ \capa_{j_0} } \binom{t+ \capa_{j_1}}{ \capa_{j_1}} E_{\capa-t e_{j_0} +( t-1) e_{j_1}} \label{eqaux}
\end{align}
so that
\begin{align*}
 f(Z'_\capa)
 &=f(E_\capa) + \sum_{j=1}^m \xi_j f(E_{\capa - e_j}) + \xi_{j_1} f( E_{\capa - e_{j_0}} ) \\
 &= \sum_{t=0} ^{ \capa_{j_0} } \binom{t+ \capa_{j_1}}{ \capa_{j_1}} \Big( E_{\capa-t e_{j_0} + t e_{j_1}} + \sum_{j\neq j_1} \xi_j E_{\capa-e_j-t e_{j_0} + t e_{j_1}}\Big) + \xi_{j_1} \Big( f( E_{\capa - e_{j_1}} ) + f( E_{\capa - e_{j_0}} )\Big) \\
 &= \sum_{t=0} ^{ \capa_{j_0} } \binom{t+ \capa_{j_1}}{ \capa_{j_1}} Z_{\capa-t e_{j_0} + t e_{j_1}}\in F, \mbox{ using Eq.~\eqref{eqaux}.}
 \end{align*}
Therefore $f(F')=F$, and we deduce the result as in Case 2. This concludes the proof of Lemma~\ref{lemrat1}.\end{proof}

\begin{lem}\label{lemrat2} Let $m,N\geq 1$. We consider the subspace
$$K = \Span \{E_\capa, \, \capa\in\Omega, \, \capa_m=0\}.$$
Assume that Theorem~\ref{thrat} holds for any $\xi\in\C^m$ and any $R$ defined over $\Qbar$ such that 
$$ \dim(F\cap R\cap K)\geq 2\dim(F\cap R) - \dim(R).$$
Then Theorem~\ref{thrat} holds for any $\xi\in\C^m$ and any $R$ defined over $\Qbar$.
\end{lem}
\begin{proof} We start with any $\xi = ( \xi_1,\ldots,\xi_m)\in\C^m$ such that $1$, $\xi_1$, \ldots, $\xi_m$ are linearly independent over $\Qbar$. Let $R$ be a subspace of $\C^\Omega$ defined over $\Qbar$, and put $\vr=\dim R$. We shall construct $g\in\GL(\C^m)$ defined over $\Qbar$ (i.e., whose matrix in the canonical basis of $\C^m$ has coefficients in $\Qbar$), such that $\xi'=g(\xi)$ satisfies $ \dim(F'\cap R'\cap K)\geq 2\dim(F'\cap R') - \dim(R')$ for any subspace $R'$ of $\C^\Omega$ defined over $\Qbar$ of dimension $\vr$, where $F'$ is associated with $\xi'=g(\xi)$ as before the statement of Theorem~\ref{thrat}. Then Lemma~\ref{lemrat1} shows that if Theorem~\ref{thrat} holds for $\xi'$, then it does for $\xi$.

Denote by $\Vz$ the Zariski closure of $\{\xi\}$ in the affine space $\mathbb A_{\Qbar}^m$ over $\Qbar$, i.e. the smallest subset of $\C^m$, defined by polynomial equations with coefficients in $\Qbar$, that contains $\xi$. In more concrete terms, $\Vz$ is the set of all $(z_1,\ldots,z_m)\in\C^m$ such that $P(z_1,\ldots,z_m)=0$ for any $P\in\Qbar[X_1,\ldots,X_m]$ such that $P(\xi_1,\ldots,\xi_m)=0$. Since $\xi_1$ is transcendental, $\Vz$ has dimension at least 1. There exists an algebraic curve $\calC$, defined by polynomial equations with coefficients in $\Qbar$, contained in $\Vz$ but in no hypersurface defined over $\Qbar $ of degree less than or equal to $\theta$ (with possible exceptions for such hypersurfaces that contain $\Vz$). 

There exists $j_0\in\unm$ such that the $j_0$-th projection $\calC\to\C$, $(\chi_1,\ldots,\chi_m)\mapsto \chi_{j_0}$, has infinite image; then this image contains all real numbers greater than some $M_0$. Parametrizing a branch of $\calC$, we obtain algebraic functions over $\Qbar(z)$, denoted by $\chi_1(z),\ldots,\chi_m(z)$, such that $ ( \chi_1(a),\ldots,\chi_m(a))\in\calC$ for any real $a\geq M_0$, and 
$| \chi_{j_0}(a)|\to\infty$ as $a\to+\infty$ with $a\in\R$. Their asymptotic behaviour as $a\to+\infty$ is given by $ \chi_{j}(a) \sim \varpi_j^0 a^{d_j}$ with $\varpi_j^0\in\Qbar\etoile$ and $d_j\in\Q$; we have $d_{j_0}>0$. Let $D= \max(d_1,\ldots,d_m)>0$, and put $ \varpi_j = \varpi_j^0$ for any $j\in\unm$ such that $d_j=D$, and $ \varpi_j = 0$ otherwise. In this way, for any $j\in\unm$ we have
\begin{equation}\label{eqrat1}
 \varpi_j = \lim_{a\in\R, \, a\to+\infty} a^{-D} \chi_j(a) \quad \mbox{ with } D>0.
 \end{equation}

We can now construct a bijective linear map $g:\C^m\to\C^m$ defined over $\Qbar$ such that $g(\varpi_1,\ldots,\varpi_m)=(0,\ldots,0,1)$. We let $\xi'=(\xi'_1,\ldots,\xi'_m) = g(\xi)$; then $1$, $\xi'_1$, \ldots, $\xi'_m$ are linearly independent over $\Qbar$. We denote by $F'$, $\calC'$, $\varpi'_j$, \ldots the objects defined as above, starting from $\xi'$ instead of $\xi$. Then we may choose $\calC'=g(\calC)$, $ ( \chi'_1(z),\ldots,\chi'_m(z))=g ( \chi_1(z),\ldots,\chi_m(z))$ so that $D'=D$ and $(\varpi'_1,\ldots,\varpi'_m) = g(\varpi_1,\ldots,\varpi_m)=(0,\ldots,0,1)$. As explained at the beginning of the proof, we shall prove that $\xi'$ satisfies the additional property $ \dim(F\cap R\cap K)\geq 2\dim(F\cap R) - \dim(R)$ for any subspace $R$ of $\C^\Omega$ of dimension $\vr$ defined over $\Qbar$; here and below (until the end of the proof), for simplicity we write $\xi$, $\calC$, $\Vz$, $F$, $R$, $\varpi_j$, \ldots instead of $\xi'$, $\calC'$, $\Vz'$, $F'$, $R'$, $\varpi'_j$, \ldots.

We let $R$ be a subspace of $\C^\Omega$ of dimension $\vr$ defined over $\Qbar$, write 
$$ 
d = \dim(F\cap R), \quad \quad
d' = \dim(F\cap R \cap K),$$
and assume (by contradiction) that $d'<2d-\vr$. 

For any $\chi = (\chi_1 ,\ldots,\chi_m ) \in\C^m$, we denote by $F_\chi$ the subspace defined exactly like $F$, except that $\chi_1$, \ldots, $\chi_m$ are used instead of $\xi_1$, \ldots, $\xi_m$ to define the $Z_\capa$. We denote by $\calV$ the set of all $\chi\in\C^m$ such that 
$$ \dim F_\chi = \theta , \quad
 \dim(F_\chi\cap R) = d , \quad
 \dim(F_\chi\cap R \cap K)=d' ,$$
i.e. such that these dimensions are the same as for $\chi=\xi$. 

We claim that we have inclusions
\begin{equation}\label{eqratinclu}
(\C^m\setminus H_1)\cap (\C^m\setminus H_2)\cap (\C^m\setminus H_3)\cap H_4\cap\ldots\cap H_v
\subset \calV \subset H_4\cap\ldots\cap H_v
\end{equation}
where $v\geq 3$ and $H_1$, \ldots, $H_v$ are hypersurfaces of $\C^m$ of degree at most $\theta$ defined over $\Qbar$, such that $\Vz\not\subset H_i$ for any $i\in\untrois$. We recall that $\Vz$ is the Zariski closure of $\{\xi\}$ in $\mathbb A_{\Qbar}^m$, so that $\Vz\not\subset H_i$ is equivalent to $\xi\not\in H_i$. Indeed we denote by $Z(X_1,\ldots,X_m) \in M_{\omega,\theta}(\Qbar[X_1,\ldots,X_m])$ the matrix of which the columns are the coordinates of the $Z_\capa$ in the canonical basis of $\C^\Omega$ (for $\capa\in\Theta$), in which $\xi_j$ is replaced with $X_j$. Then $\dim F_\chi = \rk(Z(\chi))$ is equal to $\theta$ if, and only if, at least one minor of $Z(\chi)$ of size $\theta$ is non-zero. The coefficients of $Z(X_1,\ldots,X_m) $ are polynomials of total degree at most 1 in $X_1$, \ldots, $X_m$, so each minor of size $\theta$ has degree at most $\theta$. We choose a minor which is non-zero at $\xi$, and denote by $H_1$ the hypersurface defined in $\C^m$ by the vanishing of this minor. 

Now we consider the matrix $S(X_1,\ldots,X_m) \in M_{\omega,\theta+\vr}(\Qbar[X_1,\ldots,X_m])$ of which the first $\theta$ columns are those of 
 $Z(X_1,\ldots,X_m)$, and the last $\vr$ columns belong to $\Qbar^\Omega$ and make up a basis of $R$ (which is possible since $R$ is defined over $\Qbar$). Assuming that $\dim F_\chi=\theta$, we have $\dim(F_\chi\cap R)=d$ if, and only if, $\dim(F_\chi+R) = \theta+\vr-d$; this is equivalent to $\rk (S(\chi) ) = \theta+\vr-d$. This condition can be expressed as the vanishing of all minors of size $ \theta+\vr-d+1$, and the non-vanishing of at least one minor of size $ \theta+\vr-d$. Again we choose such a minor of size $ \theta+\vr-d$ that does not vanish at $\xi$, and denote by $H_2$ the corresponding hypersurface (which has degree at most $\theta$); we define $H_4$, $H_5$, \ldots to be the hypersurfaces defined by the vanishing of the minors of size $ \theta+\vr-d+1$. We proceed in the same way with $R\cap K$ instead of $R$ to ensure that $ \dim(F_\chi\cap R \cap K)=d'$. This concludes the proof of the claimed inclusions~\eqref{eqratinclu}.

 \medskip
 
These inclusions imply that $\calC\setminus\calC_1\subset\calV$ for some finite set $\calC_1$. Indeed, we have $\xi\in\calV\subset H_i$ for any $i\in\quatrev$, so that $\calC\subset\Vz\subset H_i$ by definition of $\Vz$, since $H_i$ is defined over $\Qbar$. For $i\in\untrois$, we have $\Vz\not\subset H_i$ and $H_i$ is a hypersurface of degree at most $\theta$, so that $\calC\not\subset H_i$ (by construction of $\calC$) and $\calC\cap H_i$ is a finite set; taking for $\calC_1$ the union of these finite sets, Eq.~\eqref{eqratinclu} yields $\calC\setminus\calC_1\subset\calV$.

Since $\calC_1 $ is finite, there exists a real number $M_1\geq M_0$ such that for any real $a\geq M_1$, the point $\chi(a) = (\chi_1(a) ,\ldots,\chi_m (a))$ belongs to $\calC\setminus\calC_1\subset\calV$. We shall focus on real algebraic values of $a\geq M_1$; since $\chi(a)\in \calV\cap\Qbar^m$, the subspace $F_{\chi(a)}$ is then defined over $\Qbar$ and we have $\dim F_{\chi(a)}= \theta$, $ \dim(F_{\chi(a)}\cap R) = d $, 
$\dim(F_{\chi(a)}\cap R \cap K)=d'$. We fix such an $a$, denoted by $a_0$, and consider a subspace $W$ defined over $\Qbar$ such that 
\begin{equation}\label{eqratw}
 \big( F_{\chi(a_0)}\cap R \cap K \big) \oplus W = F_{\chi(a_0)}\cap R.
 \end{equation}
For any $a\in\Qbar\cap \R$ such that $a\geq M_1$, we have $\chi(a)\in \calV$ so that 
$$ \dim W + \dim\big( F_{\chi(a)}\cap R\big) = (d-d')+d>\vr = \dim R$$
since we have assumed (by contradiction) that $d'<2d-\vr$. Accordingly these subspaces of $R$ are not in direct sum:
 there exists $u_a\in W\cap F_{\chi(a)}\cap R$ with $u_a\neq 0$. Since $W\cap F_{\chi(a)}\cap R$ is defined over $\Qbar$, we may assume that 
$u_a\in\Qbar^\Omega$. We write $u_a= (u_{a,\capa})_{\capa\in\Omega}$ with $ u_{a,\capa}\in\Qbar$. These coordinates satisfy
\begin{equation}\label{eqrat3}
u_{a,\capa} = \sum_{j=1}^m \chi_j(a) u_{a, \capa+e_j} \mbox{ for any } \capa\in\Omega\setminus\Theta
\end{equation}
since all generators of $ F_{\chi(a)}$ satisfy these linear equations. 

Let $(T_1,\ldots,T_w)$ be a basis of $W$ consisting of vectors of $\Qbar^\Omega$, with $w=\dim W=d-d'$. Since $u_a\in W\cap \Qbar^\Omega$ there exist $\lambda_{a,1}, \ldots,\lambda_{a,w}\in\Qbar$, not all zero, such that $u_a=\sum_{\ell=1}^w \lambda_{a,\ell} T_\ell$. Writing $T_\ell = (t_{\ell,\capa})_{\capa\in\Omega}$ we have $u_{a,\capa} = \sum_{\ell=1}^w \lambda_{a,\ell} t_{\ell,\capa}$ for any $\capa\in\Omega$. Using this into Eq.~\eqref{eqrat3} yields
\begin{equation}\label{eqrat4}
 \sum_{\ell=1}^w \lambda_{a,\ell} P_{\ell,\capa}(a)=0 \mbox{ for any } \capa\in\Omega\setminus\Theta,
\end{equation}
where
$$
P_{\ell,\capa}(z)= - t_{\ell,\capa} + \sum_{j=1}^m t_{\ell,\capa+e_j} \chi_j(z) 
$$
is a function algebraic over $\Qbar(z)$. Let $P(z)$ denote the matrix $(P_{\ell,\capa}(z))_{ \capa\in\Omega\setminus\Theta, 1\leq \ell \leq w}$.
For any $a\in\Qbar\cap \R$ with $a\geq M_1$, Eq.~\eqref{eqrat4} shows that $\rk(P(a))<w$: all minors of size $w$ of the matrix $P(z)$ vanish at $a$. Since these minors are functions algebraic over $\Qbar(z)$, they are identically zero: $P(z)$ has rank at most $w-1$, as a matrix with coefficients in $\overline{\mathbb Q(z)}$. This provides algebraic functions $\mu_1(z), \ldots, \mu_w(z) \in \overline{\mathbb Q(z)}$, not all zero, such that $\sum_{\ell=1}^w \mu_\ell(z) P_{\ell,\capa}(z)=0$ for any $ \capa\in\Omega\setminus\Theta$. In other words, 
\begin{equation}\label{eqrat6}
 \sum_{\ell=1}^w \mu_\ell(z) \Big(- t_{\ell,\capa} + \sum_{j=1}^m t_{\ell,\capa+e_j} \chi_j(z) \Big)=0 \mbox{ for any } \capa\in\Omega\setminus\Theta.
\end{equation}
Now as $z\to+\infty$ with $z\in\R$, each non-zero function $\mu_\ell(z)$ has an asymptotic behavior given by $\mu_\ell(z)\sim \mu_{\ell,0}z^{e_\ell}$ with $\mu_{\ell,0}\in\Qbar\etoile$ and $e_\ell \in \Q$; if $\mu_\ell(z)$ is identically zero we put $e_\ell=-\infty$. Let $e=\max(e_1,\ldots,e_m)$; for any $\ell$, we let $\mu_{\ell,1}=\mu_{\ell,0}$ if $e_\ell=e$, and $\mu_{\ell,1}=0$ otherwise, so that $\lim_{z\to+\infty} z^{-e} \mu_\ell(z) = \mu_{\ell,1}$. We recall that Eq.~\eqref{eqrat1} has a similar flavour, and that in this equation we have $\varpi_1=\ldots=\varpi_{m-1}=0$, $\varpi_m=1$, since $\xi'$ (denoted now by $\xi$) has been constructed for this purpose. Combining these limits and letting $z\to+\infty$, with $z\in\R$, Eq.~\eqref{eqrat6} yields 
(since $D>0$)
$$ \sum_{\ell=1}^w \mu_{\ell,1} t_{\ell,\capa+e_m} = 0 \mbox{ for any } \capa\in\Omega\setminus\Theta.
$$
Let $T = \sum_{\ell=1}^w \mu_{\ell,1} T_\ell\in W\setminus\{0\}$, and write $T = (t_{ \capa})_{\capa\in\Omega}$. Then we have $t_{\capa+e_m}=0$ for any $ \capa\in\Omega\setminus\Theta$, i.e. $t_\lambda=0$ for any $\lambda\in\Theta$ such that $\lambda_m\geq 1$. For any $\lambda\in\Omega\setminus\Theta$ such that $\lambda_m\geq 1$, we obtain $t_\lambda=\sum_{j=1}^m \chi_j(a_0)t_{\lambda+e_j} = 0$ since $T\in W \subset F_{\chi(a_0}$ (see Eq.~\eqref{eqrat3}). Therefore $T\in K$: this is a contradiction with the definition~\eqref{eqratw} of $W$. This concludes the proof of Lemma~\ref{lemrat2}.\end{proof}

\subsection{Proof of Theorem~\ref{thrat}}\label{ssec:proofthmthrat}

We prove Theorem~\ref{thrat} by induction on $m+N$. Letting $m,N\geq 1$, 
 if $m\geq 2$ (resp. $N\geq 2$) we may assume that Theorem~\ref{thrat} holds with $m-1$ instead of $m$ (resp. $N-1$ instead of $N$). We shall apply this idea using the linear map $\pi:\C^\Omega\to\C^\Ombar$ defined by $\pi(E_\capa)=E_{\capa-e_m}$ for any $\capa\in\Omega$; here we let 
$$\Ombar = \{\capa\in\N^m, \, \inegm\}.$$
As explained at the beginning of \S \ref{sectech}, we have $E_{\capa-e_m}=0$ if, and only if, $\capa-e_m$ has a negative coordinate, i.e. $\capa_m=0$. Therefore the kernel of $\pi$ is 
$$K =\ker \pi = \Span \{E_\capa, \, \capa\in\Omega, \, \capa_m=0\} $$
with the same notation $K$ as in Lemma~\ref{lemrat2}. 

\medskip

The sketch of the proof is the following. Using Theorem~\ref{thrat} with $m-1$ if $m\geq 2$, we shall prove that 
\begin{equation} \label{eqthrat1}
\dim(R\cap K) \geq \Big( 2 - \frac{m-2}{N+m-2}\Big) \dim(F\cap R\cap K).
\end{equation}
Then we shall use Theorem~\ref{thrat} with $N-1$ (if $N\geq 2$) to prove that 
\begin{equation} \label{eqthrat2}
\dim\pi(R ) \geq \Big( 2 - \frac{m-1}{N+m-2}\Big) \dim\pi(F\cap R).
\end{equation}
At last we shall conclude the proof by combining these inequalities with the one provided by Lemma~\ref{lemrat2}. 

\medskip

Let us start by proving Eq.~\eqref{eqthrat1}. It holds trivially if $m=1$ since in this case, $K=\{0\}$. Therefore we may assume $m\geq 2$ and let 
$\Omega' = \{\capa\in\N^{m-1}, \, \ineg\}$, $\xi'=(\xi_1,\ldots,\xi_{m-1})$, and $F'$ be the subspace defined from $(m-1,N,\xi')$ as explained before the statement of Theorem~\ref{thrat}, spanned by vectors $Z'_\capa\in\C^{\Omega'}$ for $\capa\in\Theta'=\{\capa\in\N^{m-1}, \, |\capa| = N\}.$

Let $\iota$ denote the injective linear map $\C^{\Omega'}\to\C^\Omega$ defined by $\iota(E_\capa)=E_{(\capa,0)}$ for any $\capa=(\capa_1,\ldots,\capa_{m-1})\in\Omega'$, where $
(\capa,0)$ stands for $ (\capa_1,\ldots,\capa_{m-1},0)$. We claim that $\iota(F')=F\cap K$. The inclusion $\iota(F')\subset F\cap K$ is obvious since $\iota(Z'_\capa)=Z_{(\capa,0)}$ for any $\capa\in\Theta'$. Conversely, let $Z = \sum_{\capa\in\Theta} \lambda_\capa Z_\capa\in F\cap K$, with complex numbers $ \lambda_\capa$. For any $\capa\in\Theta$, the coordinate of $Z$ on $E_\capa$ (in the canonical basis of $\C^\Omega$) is equal to $ \lambda_\capa$. Since $Z\in K$, we deduce that $ \lambda_\capa=0$ for any $ \capa\in\Theta$ such that $\capa_m\geq 1$. Therefore $Z=\iota(\sum_{\capa\in\Theta'} \lambda_{(\capa,0)} Z'_\capa)\in \iota(F')$ and the claim follows. 

We may apply Theorem~\ref{thrat} with $(m-1,N,\xi')$ to the subspace $\iota^{-1}(R)$ which is defined over $\Qbar$. This yields 
\begin{equation} \label{eqthrat3}
\dim\iota^{-1}(R ) \geq \Big( 2 - \frac{m-2}{N+m-2}\Big) \dim(F'\cap \iota^{-1}(R ) ).
\end{equation}
Since $\iota(F')=F\cap K$ we have $ F'\cap \iota^{-1}(R ) = \iota^{-1}(F\cap R\cap K) $. Now the image of $\iota$ is $K$, so that $\dim(F'\cap \iota^{-1}(R ) ) = \dim( F\cap R\cap K) $ and $\dim\iota^{-1}(R ) = \dim(R\cap K) $: Eq.~\eqref{eqthrat1} follows from Eq.~\eqref{eqthrat3}.

\medskip

We shall now prove Eq.~\eqref{eqthrat2}. If $N=1$ it reads $\dim\pi(R ) \geq \dim\pi(F\cap R)$ and holds trivially. Let us assume $N\geq 2$. Recall that $\Ombar = \{\capa\in\N^m, \, \inegm\}$, and denote by $\overline Z_\capa$, for $\capa\in\overline\Theta=\{\capa\in\N^m, \, |\capa|=N-1\}$, the vectors constructed from $\xi$ with respect to $m$ and $N-1$. We write $\overline F=\Span\{ \overline Z_\capa, \, \capa\in\overline\Theta\}$. It is clear that for any $\capa\in \Theta$ we have $\pi(Z_\capa)=\overline Z_{\capa-e_m}$ if $\capa_m\geq 1$, and $\pi(Z_\capa)=0$ otherwise. Therefore $\pi(F)=\overline F$, and Theorem~\ref{thrat} applied to $\pi(R)$ with $(m ,N-1,\xi )$ yields
\begin{equation} \label{eqthrat4}
\dim\pi(R ) \geq \Big( 2 - \frac{m-1}{N+m-2}\Big) \dim(\pi(F) \cap \pi(R ) ).
\end{equation}
Since $\pi(F\cap R)\subset \pi(F) \cap \pi(R )$ this implies Eq.~\eqref{eqthrat2}.

\bigskip

We may now conclude the proof of Theorem~\ref{thrat}, since using Lemma~\ref{lemrat2} we may assume that 
\begin{equation} \label{eqthrat4bis}
\dim R \geq 2 \dim( F \cap R ) - \dim (F\cap R\cap K ).
\end{equation}
The restriction $\pi_R$ has kernel $R\cap K$ and image $\pi(R)$, so that $\dim R =\dim(R\cap K)+\dim\pi(R)$. Similarly, $\dim (F\cap R) =\dim(F\cap R\cap K)+\dim\pi(F\cap R)$. Therefore adding Eqs.~\eqref{eqthrat1} and~\eqref{eqthrat2} yields
$$
\dim R \geq \Big( 2 - \frac{m-1}{N+m-2}\Big) \dim(F \cap R ) + \frac{1}{N+m-2} \dim(F\cap R\cap K).
$$
Multiplying this equation by $\frac{N+m-2}{N+m-1}$, and adding Eq.~\eqref{eqthrat4bis} divided by $N+m-1$, yields
\begin{equation} \label{eqfin}
\dim R \geq \Big( 2 - \frac{m-1}{N+m-1}\Big) \dim(F \cap R ) .
\end{equation}
This concludes the proof of the inequality in Theorem~\ref{thrat}.

\medskip

Assume now that equality holds in Eq~\eqref{eqfin}. If $m=1$ then $\dim F = \theta=1$, $\omega=2$, and $R$ is either $\{0\}$ or $\C^{\Omega}$ since $F$ is not defined over $\Qbar$ (because $\xi_1$ is transcendental). Assume now that $m\geq 2$, and notice that equality holds in Eqs.~\eqref{eqthrat1},~\eqref{eqthrat2},~\eqref{eqthrat4} and~\eqref{eqthrat4bis}. Using Theorem~\ref{thrat} with $m-1$ instead of $m$, since equality holds in Eq.~\eqref{eqthrat1} we have either $R\cap K = \{0\}$ or $R\cap K = K$. In the former case, equality in Eq.~\eqref{eqthrat4bis} implies $ \dim R = 2 \dim(F\cap R)$, which implies $\dim R = 0 $ since we have assumed that equality holds in Eq~\eqref{eqfin}. In the latter case, we use (if $N\geq 2$) Theorem~\ref{thrat} and the fact that equality holds in Eq.~\eqref{eqthrat4} to deduce a new 
alternative: either $\pi(R)=\C^{\Ombar}$ or $\pi(R)=\{0\}$. In the former case, since $R\cap K = K$ we obtain $R= \C^\Omega$. In the latter case, we have $R= K$ so that Eq.~\eqref{eqthrat1} reads
 $\dim R = \Big( 2 - \frac{m-2}{N+m-2}\Big) \dim(F\cap R) >0$. This contradiction with equality in Eq.~\eqref{eqfin} concludes the proof of Theorem~\ref{thrat}.

\noindent St\'ephane Fischler, Universit\'e Paris-Saclay, CNRS, Laboratoire de math\'ematiques d'Orsay, 91405 Orsay, France.

\medskip

\noindent Tanguy Rivoal, Universit\'e Grenoble Alpes, CNRS, Institut Fourier, CS 40700, 38058 Grenoble cedex 9, France.

\bigskip

\noindent Keywords: $E$-functions, Irrationality Measures, Graded Pad\'e Approximants, Roth's Theorem.

\bigskip

\noindent MSC 2020: 11J82 (Primary), 11J91 (Secondary)


\begin{thebibliography}{10}

\bibitem{Borisdec} B. Adamczewski, {On the
expansion of some exponential periods in an integer base}, {\em Math. Ann.} {\bf 346}
(2010), 107--116.

\bibitem{bori} B. Adamczewski, T. Rivoal, {Exceptional values of $E$-functions at algebraic points}, {\em Bull. Lond. Math. Soc.} {\bf 50}.4 (2018), 697--708. 

\bibitem{andre} Y. Andr\'e, S\'eries Gevrey de type arithm\'etique I. Th\'eor\`emes de puret\'e et de dualit\'e,
{\em Ann. of Math.} {\bf 151}.2 (2000), 705--740.

\bibitem{andre2} Y. Andr\'e, S\'eries Gevrey de type arithm\'etique II. Transcendance sans transcendance, {\em Ann. of Math.} {\bf 151}.2 (2000), 741--756.

\bibitem{bakerexp} A. Baker, On some Diophantine inequalities involving the exponential function, {\em Can. J. Math.} {\bf 17} (1965), 
616--626. 

\bibitem{DBShid} D.~Bertrand, Le th\'eor\`eme de {S}iegel-{S}hidlovsky
 revisit\'e, in {\em Number theory, Analysis and Geometry: in memory of
 Serge Lang} (D. Goldfeld {\em et al.}, eds), Springer, 2012,
 51--67.
 
\bibitem{BB} D.~Bertrand, F.~Beukers, {\'{E}quations diff\'erentielles lin\'eaires et majorations de multiplicit\'es}, {\em Ann. Sci. \'Ecole Norm. Sup.} {\bf 18}.1
 (1985), 181--192.

\bibitem{BCY} D.~Bertrand, V. Chirskii, J. Yebbou, Effective estimates for global relations on Euler-type series,
{\em Annales de la faculté des sciences de Toulouse} {\bf 13}.2 (2004), 241--260.

\bibitem{beukers} F.~Beukers, {A refined version of the Siegel-Shidlovskii theorem}, {\em Ann. of Math.} {\bf 163}.1 (2006), 369--379. 

\bibitem{brs} A. Bostan, T. Rivoal, B. Salvy, Minimization of differential equations and algebraic values of $E$-functions,
{\em Math. Comp.} {\bf 93} (2024), 1427--1472.

\bibitem{chud} G. V. Chudnovsky, {On some applications of diophantine approximations}, 
{\em Proc. Nat. Acad. Sci. USA} {\bf 81} (1984), 1926--1930.

\bibitem{feldnest} N. I. Feldman, Yu. V. Nesterenko, {\em Transcendental Numbers}, in Encyclopaedia of Mathematical Sciences, Vol. {\bf 44}: Number Theory IV (Springer, 1998).

\bibitem{SFcaract} S. Fischler, Shidlovsky's multiplicity estimate and irrationality of zeta values, {\em J. Austral. Math. Soc.} {\bf 105}.2 (2018), 145--172.

\bibitem{gvalues} S. Fischler, T. Rivoal, On the values of $G$-functions, {\em Commentarii Math. Helv.} {\bf 89}.2 (2014), 313--341.

\bibitem{ateo} S. Fischler, T. Rivoal, Arithmetic theory of $E$-operators, {\em J. \'Ec. Polytech. -- Math\'e\-ma\-ti\-ques} {\bf 3} (2016), 31--65. 


\bibitem{gfndio2} S. Fischler, T. Rivoal, Linear independence of values of $G$-functions, II: outside the disk of convergence, 
{\em Ann. Math. Qu\'e.} {\bf 45} (2021), 53--93.

\bibitem{firimathz} S. Fischler, T. Rivoal, Effective algebraic independence of values of $E$-functions, {\em Math. Z.} {\bf 305} (2023), article 48. 

\bibitem{firiliouville} S. Fischler, T. Rivoal, Values of $E$-functions are not Liouville numbers, {\em J. \'Ec. Polytech. -- Ma\-th\'e\-matiques} {\bf 11} (2024), 1--18.

\bibitem{fresanlivre} J. Fres\'an, Une introduction aux p\'eriodes, in 
{\em P\'eriodes et transcendance, Journ\'ees X-UPS 2019}, 147 pages, to appear at Éditions de l'École polytechnique, Palaiseau.

\bibitem{frejos} J. Fres\'an, P. Jossen, A non-hypergeometric $E$-function, {\em Ann. of Math.} {\bf 194}.3 (2021), 903--942.

\bibitem{Joseph1} E. Joseph, On the approximation exponents for subspaces of $\R^n$,
{\em Mosc. J. Comb. Number Theory} {\bf 11}.1 (2022), 21--35.

 \bibitem{Joseph2} E. Joseph, Upper bounds and spectrum for approximation exponents for subspaces of $\R^n$, {\em J. Théor. Nombres Bordeaux} {\bf 34}.3 (2022), 827--850.

\bibitem{Kappe} L.-C. Kappe, Zur Approximation von $e^\alpha$, {\em Ann. Univ. Sci. Budap. Rolando E\"otv\"os, Sect. Math.} {\bf 9} (1966), 3--14.

\bibitem{langbsmf} S. Lang, Report on Diophantine approximations, {\em Bull. Soc. Math. France} {\bf 93} (1965), 177--192.

\bibitem{lepetit} G. Lepetit, Le th\'eor\`eme d'Andr\'e-Chudnovsky-Katz au sens large, {\em North-West. Eur. J. Math.} {\bf 7} (2021), 83--149.

 \bibitem{marco} R. Marcovecchio, The Rhin-Viola method for $\log(2)$, 
{\em Acta Arith.} {\bf 139}.2 (2009), 147--184. 

\bibitem{MW} D. W. Masser, G. Wüstholz, Zero estimates on group varieties I, {\em Invent. Math.} {\bf 64} (1981), 489--516.

\bibitem{ns} Yu. V. Nesterenko, A. B. Shidlovskii, On the linear independence of values of $E$-functions, {\em Sb. Math.} {\bf 187} (1996), 1197--1211; translation from the russian {\em Math. Sb.} {\bf 187} (1996), 93--108.

\bibitem{Pph} P. Philippon, Lemmes de zéros dans les groupes algébriques
commutatifs, {\em Bull. Soc. Math. France} {\bf 114} (1986), 355--383.

\bibitem{poole} E. G. C. Poole, {\em Introduction to the theory of linear differential equations}, Dover Publications Inc., New York, 1960.


\bibitem{roth} K. F. Roth, Rational approximations to algebraic numbers, {\em Mathematika} {\bf 2}.1 (1955), 1--20, corrigendum p.~168.

\bibitem{Shidlovskiilivre} A.~B. Shidlovskii, {\em Transcendental numbers}, de Gruyter Studies in Math. {\bf 12}, de Gruyter, Berlin, 1989.

\bibitem{siegel} C. Siegel, \"Uber einige Anwendungen diophantischer Approximationen, vol. 1 S. {\em Abhandlungen Akad.}, Berlin, 1929.

\bibitem{putsinger} M. van der Put, M. F. Singer, {\em Galois theory of linear differential equations}, Grundlehren der Mathematischen Wissenschaften, vol. {\bf 328}, Springer-Verlag, Berlin, 2003.

\bibitem{Asterisque} M. Waldschmidt, {\em Nombres transcendants et groupes alg\'ebriques}, with appendices by D. Bertrand and J.-P. Serre, Ast\'erisque, vol. {\bf 69--70}, Soc. Math. France, 1979.


\bibitem{Wu} G. Wüstholz, Multiplicity Estimates on Group Varieties, {\em Ann. of Math.} {\bf 129}.3 (1989), 471--500.

\bibitem{zz} D. Zeilberger, W. Zudilin, The irrationality measure of $\pi$ 
is at most $7.103205334137\ldots$, 
{\em Mosc. J. Comb. Number Theory} {\bf 9}.4 (2020), 407--419.


\bibitem{Zudilin} W. Zudilin, On rational approximations of values of a certain class of entire functions, 
{\em Sb. Math.} {\bf 186}.4 (1995), 555--590; translation from the russian {\em Mat. Sb.} {\bf 186}.4 (1995), 89--124. 

\end{thebibliography}
\end{document}